\documentclass[11pt,reqno]{amsart}
\usepackage{amssymb}
\usepackage{enumerate}
\usepackage{amsmath,amssymb,amsfonts,amsthm,
	latexsym, amscd, epsfig, epic,color}
\usepackage{mathtools} 
\usepackage{extarrows}
\usepackage{enumitem}
\usepackage{enumerate}
\usepackage{mathrsfs}
\usepackage{eucal}%    caligraphic-euler fonts: \mathcal{ }
\usepackage{eufrak}%   frak-euler        fonts: \mathfrak{ }
\usepackage{xspace}
\usepackage{a4wide}
\usepackage{colordvi}
\usepackage{comment}
\usepackage{hyperref,eucal}
\usepackage{tikz}\usetikzlibrary{matrix,arrows}
\newcounter{mycount}
\theoremstyle{plain}

\newtheorem{theorem}[mycount]{Theorem}
\newtheorem{corollary}[mycount]{Corollary}

\newtheorem{lemma}[mycount]{Lemma}

\newtheorem{conjecture}[mycount]{Conjecture}
\theoremstyle{definition}
\newtheorem{definition}{Definition}

\theoremstyle{example}

\theoremstyle{openproblem}

\def\T{\CMcal{T}}
\def\M{\CMcal{M}}

\def\R{\CMcal{R}}

\def\A{\CMcal{A}}

\theoremstyle{remark}
\newtheorem{remark}{Remark}
\numberwithin{equation}{section}
\numberwithin{figure}{section}
\def\des{\mathsf{des}}
\def\ides{\mathsf{ides}}
\def\iasc{\mathsf{iasc}}
\newcommand{\asc}{\mathsf{asc}}
\def\max{\mathsf{max}}
\def\zero{\mathsf{zero}}
\def\rep{\mathsf{rep}}
\def\rmin{\mathsf{rmin}}
\def\lmax{\mathsf{lmax}}
\def\lmin{\mathsf{lmin}}
\def\rmax{\mathsf{rmax}}

\def\rloc{\mathsf{rloc}}
\def\dist{\mathsf{dist}}

\def\fmax{\mathsf{fmax}}
\def\cmax{\mathsf{cmax}}
\def\ealm{\mathsf{ealm}}
\def\ealz{\mathsf{ealz}}

\def\mrep{\mathsf{mrep}}

\def\A{\mathcal{A}}
\def\T{\CMcal{T}}
\def\I{\CMcal{I}}
\def\CS{\CMcal{S}}

\def\C{ \CMcal{C}}
\def\P{ \CMcal{P}}

\def\min{\mathrm{min}}
\def\des{\mathsf{des}}
\def\ides{\mathsf{ides}}
\def\iasc{\mathsf{iasc}}
\def\asc{\mathsf{asc}}
\def\max{\mathsf{max}}
\def\zero{\mathsf{zero}}
\def\rep{\mathsf{rep}}
\def\rmin{\mathsf{rmin}}
\def\lmax{\mathsf{lmax}}
\def\lmin{\mathsf{lmin}}
\def\rmax{\mathsf{rmax}}
\def\cmax{\mathsf{cmax}}
\def\czero{\mathsf{czero}}

\def\Rmin{\mathsf{Rmin}}

\def\ealm{\mathsf{ealm}}

\def\rpos{\mathsf{rpos}}
\def\S{\mathfrak{S}}

\title[Symmetric generating functions and Euler-Stirling statistics]{Symmetric generating functions and Euler-Stirling statistics on permutations}
\author{Emma Yu Jin}
\thanks{The author was supported by the Austrian Research Fund FWF Elise-Richter Project V 898-N, is supported by the Fundamental Research Funds for the Central Universities, Project No. ZK1124 and the National Nature Science Foundation of China (NSFC), Project No. 12201529.}
\address{School of Mathematical Sciences, Xiamen University, Xiamen 361005, P.R. China}
\email{yjin@xmu.edu.cn}

\date{\today}
\begin{document}
\begin{abstract}
We present (bi-)symmetric generating functions for the joint distributions of  Euler-Stirling statistics on permutations, including the number of descents ($\des$), inverse descents ($\ides$), the number of left-to-right maxima ($\lmax$), the number of right-to-left maxima ($\rmax$) and the number of left-to-right minima ($\lmin$). We also show how they recover the classical symmetric generating function of permutations due to Carlitz, Roselle and Scoville (1966).

Our proofs exploit three different recursive constructions of inversion sequences, bijections on the multiple equidistributions of Euler-Stirling statistics over permutations and transformation formulas of basic hypergeometric series. Furthermore, we establish a new quadruple equidistribution of Euler-Stirling statistics over inversion sequences, as progress towards a conjecture proposed by Schlosser and the author (2020).
\end{abstract}

%%%%%%%%%%%%%%%%%%%%%%
%%%%%%%%%%%%%%%%%%%%%%
	\maketitle
\section{Introduction and Main results}
Through generating functions, the study of permutations statistics has remarkably diverse connections to other mathematical areas such as symmetric functions \cite{gr:93,Stanley:altper,Stanley:book}, basic hypergeometric series \cite{aj,BLR:14,fjlyz,js} and probabilistic distributions \cite{Flajolet:book}.

The Eulerian and Stirling statistics are two kinds of classical statistics. 
Let $\S_n$ be the set of permutations of $[n]:=\{1,2,\ldots,n\}$. Given a permutation $\pi=\pi_1\pi_2\cdots\pi_n\in\S_n$, define 
\begin{align*}
\des(\pi):=|\{i\in [n-1]:\pi_i>\pi_{i+1}\}|
\end{align*}
to be the number of descents of $\pi$. The descent polynomial over permutations 
\begin{align*}
A_n(t):=\sum_{\pi\in\S_n}t^{\des(\pi)}
\end{align*}
is called the Eulerian polynomial and the coefficients $[t^k]A_n(t)$ are the {\em Eulerian numbers}. Let
\begin{align}\label{E:lmax}
\lmax(\pi):=|\{\pi_i:\pi_i>\pi_{j} \mbox{ for all }j<i\}|
\end{align}
be the number of left-to-right maxima, then 
\begin{align*}
(x)_n:=x(x+1)\cdots(x+n-1)=\sum_{\pi\in\S_n}x^{\lmax(\pi)}
\end{align*}
and the coefficients $[x^k](x)_n$ are the {\em signless Stirling numbers of the first kind}. Any statistic whose distribution over permutations equals the Eulerian numbers (resp. signless Stirling numbers of the first kind) is called an {\em Eulerian statistic} (resp. {\em Stirling statistic}). For instance, the number of excedances ($\mathsf{exc}$) and the number of inverse descent ($\ides$), defined by $\ides(\pi):=\des(\pi^{-1})$,
are Eulerian statistics. The number of cycles, the left-to-right minima ($\lmin$), and the number of right-to-left maxima ($\rmax$) are Stirling statistics, where the latter two are defined similarly to $\lmax$ (see (\ref{E:lmax})).

Recently we have explored the symmetric joint distributions of Euler-Stirling statistics over a subclass of permutations \cite{fjlyz,hj,hjs,js}. Our motivation stemmed from its bijective relations to other combinatorial structures such as $({\bf2+2})$-free posets, interval orders, certain restricted inversion sequences, Stoimenow's involution and regular linearized chord diagrams, etc (see for example \cite{bcdk,cl,dkrs,dp,fo,je,je2,kr2,kr,lev}); its connection to basic hypergeometric series \cite{aj,js}, asymptotics \cite{BLR:14,hj,hjs} and modular forms \cite{BLR:14,zag}. 

This paper is a continuation of this research line in the context of permutations. In 1966, an elegant symmetric generating function for the joint distribution of $(\des,\ides)$ was found by Carlitz, Roselle and Scoville \cite{crs}. That is,
\begin{align}\label{E:gg1}
\sum_{n=0}^{\infty}\sum_{s\in\S_n}\frac{u^{\des(\pi)}x^{\ides(\pi)}t^n}{(1-u)^{n+1}(1-x)^{n+1}}=\sum_{n=1}^{\infty}\sum_{k=1}^{\infty}\frac{x^{n-1}u^{k-1}}{(1-t)^{kn}}.
\end{align}
This identity was subsequently extended to include two Mahonian statistics (the major index and its inverse) by Garsia and Gessel \cite{gg}, employing the theory of $P$-partitions and its nice amenability to record permutation statistics. 

Returning to (\ref{E:gg1}), if we add Stirling statistics such as $\lmax,\lmin$ and $\rmax$ to the LHS of (\ref{E:gg1}), what would the generating functions look like and whether the symmetric distribution is preserved? The purpose of this paper is to answer these questions.

Our main results are (bi)-symmetric generating function formulas for the joint distributions of four Euler-Stirling statistics over permutations. 

Let $\CMcal{G}(t;x,u,v,q,z)$ be the generating function of permutations
counted by the number of descents (variable $u$), the number of inverse descents (variable $x$), the number of right-to-left maxima (variable $v$), the number of left-to-right minima (variable $q$), the number of left-to-right maxima (variable $z$) and the length (variable $t((1-x)(1-u))^{-1}$), namely, 
\begin{align}\label{E:G5}
\CMcal{G}(t;x,u,v,q,z)%:&=\sum_{n=1}^{\infty}t^n\sum_{s\in\I_n}
%\frac{u^{\asc(s)}x^{\dist(s)}q^{\max(s)}v^{\rmin(s)}z^{\mathsf{zero}(s)}}
%{(1-u)^n(1-x)^n}\\
:&=\sum_{n=1}^{\infty}t^n\sum_{\pi\in\S_n}
\frac{u^{\des(\pi)}x^{\ides(\pi)}q^{\lmin(\pi)}v^{\rmax(\pi)}z^{\lmax(\pi)}}
{(1-u)^n(1-x)^n}.
\end{align}
We derive a generating function formula for $\CMcal{G}(t;x,u,v,q,1)$ by recursively decomposing inversion sequences according to the initial consecutive zeros, while to deduce a formula for $\CMcal{G}(t;x,u,v,q,z)$ by our approach seems to be currently out of reach. The definitions of inversion sequences and related statistics are postponed to Section \ref{S:inv}. 
\begin{theorem}\label{T:1}
	The generating function $\CMcal{G}(t;x,u,v,q,1)$ for the joint distribution of a quadruple $(\des,\ides,\rmax,\lmin)$ of Euler-Stirling statistics on permutations equals
\begin{align}
\CMcal{G}(t;x,u,v,q,1)
\nonumber&=\frac{vt}{1-x}
\sum_{n=1}^{\infty}\frac{qx-1+(1-q)r^{n-1}}{x-r^n}
\prod_{i=1}^{n-1}
\frac{u(x-r^i-xvt)((1-qt)r^{i-1}-1)}{(r^{i}-1)(x-r^i)}\\
\nonumber&\qquad\times\bigg(1-\frac{ut(q-1)}{x-1}\sum_{n=1}^{\infty}
\frac{(x-1-xv)r^{n-1}+xv}{r^n-1}\\
\label{E:thm1}&\qquad\qquad\qquad\times \prod_{i=1}^{n-1}
\frac{u(x-r^i-xvt)((1-qt)r^{i-1}-1)}{(r^{i}-1)(x-r^i)}\bigg)^{-1},
\end{align}
where $r=1-t$. Furthermore, $\CMcal{G}(t;x,u,v,q,1)=\CMcal{G}(t;x,u,q,v,1)=\CMcal{G}(t;u,x,q,v,1)$, namely, $(\des,\ides,\rmax,\lmin)$, $(\des,\ides,\lmin,\rmax)$, $(\ides,\des,\rmax,\lmin)$ are all equidistributed over the set $\S_n$.
\end{theorem} 
Our second main result is two nice equivalent forms for $\CMcal{G}(t;x,u,v,1,1)$: one is a direct consequence of Theorem \ref{T:1}; and the other one is an intermediate result along our way to find an expression of $\CMcal{G}(t;x,u,v,1,z)$.
\begin{theorem}\label{T:adr}
	The generating function $\CMcal{G}(t;x,u,v,1,1)$ of permutations with respect to a triple $(\des,\ides,\rmax)$, or equivalently $(\des,\ides,\lmin)$,  $(\ides,\des,\lmin)$ or $(\ides,\des,\rmax)$ of Euler-Stirling statistics equals
	\begin{align}
	\label{E:adr1}\CMcal{G}(t;x,u,v,1,1)&=
	\sum_{n=1}^{\infty}\frac{vtu^{n-1}}{(1-t)^n-x}\prod_{i=1}^{n-1}
	\frac{x-(1-t)^i-xvt}{x-(1-t)^i}\\
	\label{E:adr2}&=\sum_{n=1}^{\infty}\frac{vt(1-t)^{n-1}}{u^{n}(x(1-t)^{n-1}-1)}
	\prod_{i=1}^{n}\frac{1-x(1-t)^{i-1}}
	{x(1-t)^{i-1}(vt-1)+1}.
	\end{align}
	Furthermore, $\CMcal{G}(t;x,u,v,1,1)=\CMcal{G}(t;u,x,1,v,1)=\CMcal{G}(t;x,u,1,v,1)=\CMcal{G}(t;u,x,v,1,1)$.
\end{theorem}
\begin{remark}\label{R:1}
An analytic proof of the fact $\CMcal{G}(t;x,u,v,1,1)=\CMcal{G}(t;u,x,v,1,1)$ follows from a transformation formula of non-terminating basic hypergeometric $_4\phi_3$ series in \cite{js}. Details are provided at the end of Section \ref{S:4}.
\end{remark}
Setting $v=1$ in both (\ref{E:adr1}) and (\ref{E:adr2}), we retrieve (\ref{E:gg1}).
\begin{corollary}\label{C:1}
	The generating function for the pair $(\des,\ides)$ of Eulerian statistics on permutations is given by
	\begin{align*}
	 \sum_{n=0}^{\infty}\sum_{\pi\in\S_n}\frac{u^{\des(s)}x^{\ides(\pi)}t^n}{(1-u)^{n+1}(1-x)^{n+1}}%&=\frac{1+\CMcal{G}(t;x,u,1,1,1)}{(1-x)(1-u)}\\
	&=\sum_{n=1}^{\infty}\frac{t(1-t)^{n-1}(u-1)^{-1}}
	{u^n(1-x(1-t)^{n-1})(1-x(1-t)^{n})}+\frac{1}{(u-1)(x-1)}\\
	%&=\sum_{n=1}^{\infty}\frac{tu^{n-1}(1-t)^{n-1}(1-u)^{-1}}
	%{(x-(1-t)^{n-1})(x-(1-t)^{n})}\\
	&=\sum_{n=1}^{\infty}\sum_{k=1}^{\infty}\frac{x^{n-1}u^{k-1}}{(1-t)^{kn}}
	\end{align*}
\end{corollary}
The second equation is proved in Section \ref{sec:3}. Using a dual version of the first decomposition, i.e., a recursive construction of inversion sequences by the initial strictly increasing maximal entries, we are able to establish a formula for $\CMcal{G}(t;x,u,v,1,z)$.

\begin{theorem}\label{T:4}
	The generating function $\CMcal{G}(t;x,u,v,1,z)$ of permutations counted by the quadruple $(\des,\ides,\rmax,\lmax)$ of Euler-Stirling statistics equals
	\begin{align*}
	\CMcal{G}(t;x,u,v,1,z)
	&=\sum_{n=1}^{\infty}\frac{ztvr^{n-1}(1+ux\,\T_n)}{u(xr^{n-1}-1)}
	\prod_{i=1}^{n-1}\frac{t(1-z)r^{i-1}+r^i-1}{u(r^i-1)}\\
	&\qquad\quad\times\left(1-\sum_{n=1}^{\infty}\frac{t(z-1)r^{n-1}}{u(r^n-1)}\prod_{i=1}^{n-1}\frac{t(1-z)r^{i-1}+r^i-1}{u(r^i-1)}\right)^{-1},
	\end{align*}
	where $r=1-t$ and $\T_n=r^{n-1}\CMcal{G}(t;xr^{n-1},u,v,1,1)$ given in Theorem \ref{T:adr}.
\end{theorem}
\begin{remark}\label{Remark:2}
Let $\iasc(\pi):=n-1-\ides(\pi)$, then by Proposition 30 of \cite{js}, the quadruples
$(\des,\iasc,\rmax,\lmax)$ and $(\iasc,\des,\lmax,\lmin)$ have the same distribution over $\S_n$. In other words, let $\CMcal{F}(t;u,x,v,z)=x^{-1}\CMcal{G}(tx;x^{-1},u,v,1,z)$, then $\CMcal{F}(t;u,x,v,z)$ is the generating function of permutations with respect to the quadruple $(\des,\iasc,\rmax,\lmax)$, or equivalently $(\iasc,\des,\lmax,\lmin)$.
%Let $\rep(s):=|s|-1-\dist(s)$ be the number of repeated entries of $s$. Then by Proposition 30 of \cite{js}, we know that there is a bijection $\varrho$ that satisfies $(\asc,\rep,\zero,\max)s=(\rep,\asc,\rmin,\zero)\varrho(s)$. 
%That is, the quadruples
%$(\des,\iasc,\rmax,\lmax)$ and $(\iasc,\des,\lmax,\lmin)$ have the same distribution over $\S_n$ through the bijection $\Theta$ in (\ref{E:bv}).
\end{remark}

Our next main result is proved by combining a typical decomposition of inversion sequences, namely successively add a new entry at the end; and a new bijection on inversion sequences.
%\begin{theorem}
%The quadruple of statistics $(\asc,\zero, \max, \mathsf{Rmin})$ have the same distribution as the one $(\mathsf{dist}, \zero, \max, \mathsf{Rmin})$ on inversion sequences.
%\end{theorem}

\begin{theorem}\label{T:asczeromax}
	The generating function of permutations with respect to the triple $(\des,\lmax,\lmin)$, or equivalently $(\ides,\lmax,\lmin)$ of Euler-Stirling statistics is
	\begin{align*}
	&\,\quad \sum_{n=1}^{\infty}\sum_{\pi\in \S_n} \frac{u^{\des(\pi)}z^{\lmax(\pi)}q^{\lmin(\pi)} t^n}{(1-u)^n}
	=\sum_{n=0}^{\infty}\frac{qztu^n}{1-(n-q+1)t}\prod_{i=0}^{n}
	\frac{1-(i-q+1)t}{1-(i+z)t}.
	\end{align*}
	In fact, $(\des,\lmax,\lmin,\rmax)$ and $(\ides,\lmax,\lmin,\rmax)$ have the same distribution on $\S_n$.
\end{theorem}

The rest of the paper is organized as follows: We provide in the next section some background on the bijection between permutations and inversion sequences, in order to interpret the generating function $\CMcal{G}(t;x,u,v,q,z)$ in terms of inversion sequences. The proof of Theorem \ref{T:1} and Corollary \ref{C:1} are presented in Section \ref{sec:3}. Section \ref{S:4} is concerned with proofs of Theorem \ref{T:adr} and \ref{T:4}. In Section \ref{S:5}, a new bijection on inversion sequences and Theorem \ref{T:asczeromax} are established. 
We conclude with some final remarks in Section \ref{S:final}.

\section{From permutations to inversion sequences}\label{S:inv}
The purpose of this section is to interpret the joint distribution of Euler-Stirling statistics in the language of inversion sequences.

An {\em inversion sequence} $(s_1,s_2,\ldots,s_n)$ is a sequence of non-negative integers such that for all $i$, $0\le s_i<i$. Let  $\CMcal{I}_n$ denote the set of inversion sequences of length $n$, which is in one-to-one correspondence with the set $\mathfrak{S}_n$ of permutations via a bijection by Baril and Vajnovszki \cite{bv}: There is a bijection $\Theta:\S_n\rightarrow \I_n$ satisfying that for
any $\pi\in\mathfrak{S}_n$,
\begin{align}\label{E:bv}
(\des, \ides,\lmin,\lmax,\rmax)\pi=(\asc,\dist,\max,\zero,\rmin)\Theta(\pi).
\end{align}
The statistics on the RHS are defined as follows: For any $s=(s_1,\ldots,s_n)\in\I_n$, let
\begin{align*}
\asc(s)&:=|\{i\in[n-1]: s_i<s_{i+1}\}|,\\
\dist(s)&:=\vert\{s_1,s_2,\ldots,s_n\}\vert-1,\\
\zero(s)&:=\vert\{i\in[n]: s_i=0\}\vert,\\
\max(s)&:=\vert\{i\in[n]:s_i=i-1\}\vert,\\
\rmin(s)&:=|\{s_i: s_i<s_j\text{ for all $j>i$}\}|,
\end{align*}
be the number of ascents, distinct non-zero entries, zeros, maximal entries and the number of right-to-left minima of $s$, respectively. By the classification of Eulerian and Stirling statistics, the first two are Eulerian, while the other three are Stirling. Let us also note that before (\ref{E:bv}), Foata \cite{fo} established a pioneering bijection that transforms every pair $(\asc,\dist)$ of Eulerian statistics on inversion sequences to a pair $(\des,\ides)$ on permutations.
It follows from (\ref{E:bv}) that 
\begin{align*}
\CMcal{G}(t;x,u,v,q,z)&=\sum_{n=1}^{\infty}t^n\sum_{s\in\I_n}
\frac{u^{\asc(s)}x^{\dist(s)}q^{\max(s)}v^{\rmin(s)}z^{\mathsf{zero}(s)}}
{(1-u)^n(1-x)^n}.
\end{align*}
Of these five Euler-Stirling statistics, the most difficult one to keep track of is $\dist$, because whether an entry is distinct or not depends on the entire sequence. The first two decompositions of inversion sequences have the advantage of locally recording {\em both} the numbers of distinct entries and ascents, while the third one in Section \ref{S:5} seems not.

\section{Proof of Theorem \ref{T:1} and Corollary \ref{C:1}}\label{sec:3}

Our proof strategy is to count a {\em subset} of inversion sequences, that is, the ones with zeros only appearing at the beginning; and then recover all inversion sequences with respect to the statistics $\asc,\dist,\max,\rmin$ by focusing on the ones with exclusive one zero. 

Let us introduce some necessary definitions and notations.
\begin{definition}For a sequence $s=(s_1,\ldots,s_n)\in\I_n$, let $\czero(s)$ be the number of the initial {\bf c}onsecutive {\bf zero}s of $s$, let $\ealz(s)$ be the {\bf e}ntry {\bf a}fter the {\bf l}ast {\bf z}ero of $s$ if $s$ is not ended with $0$; otherwise set $\ealz(s)=0$. For example, for $s=(0,0,2,1,3,2)$, then $\czero(s)=\ealz(s)=2$. 
\end{definition}
The statistic $\czero$ is sometimes called run in some literature (see for instance \cite{kr}). Let $\I:=\cup_{n\ge 1}\I_n$ be the set of inversion sequences, let $s\backslash \{i_1,i_2,\ldots,i_k\}:=\{s_1,\ldots,s_n\}\backslash\{i_1,i_2,\ldots,i_k\}$ denote the set of entries of $s$ other than $i_1,i_2,\ldots,i_k$. Then define 
\begin{align*}
\CMcal{I}^*:=\{s\in\CMcal{I}:\mathsf{czero}(s)=\zero(s) \mbox{ and } \min(s\backslash\{0\})=1\}
\end{align*}
as a subset of $\CMcal{I}$. Let $\CMcal{P}(t;u,x,z,w,v,q)$ be the generating function of inversion sequences from $\CMcal{I}^*$ counted by the statistics $\asc,\dist,\max,\rmin,\czero,\ealz$, namely,
\begin{align*}
\CMcal{P}(t;u,x,z,w,v,q):&=\sum_{s\in\I^*}t^{|s|}
u^{\asc(s)}x^{\dist(s)}q^{\max(s)}v^{\rmin(s)}z^{\czero(s)}w^{\czero(s)-\ealz(s)}.
\end{align*}
We next partition the set $\I^*$ into three disjoint subsets:
\begin{align*}
\CMcal{P}_{1}&:=\{s\in\CMcal{I}^*: |s|=\mathsf{czero}(s)+1\},\\
\CMcal{P}_{2}&:=\{s\in\CMcal{I}^*\,\backslash \CMcal{P}_1: \mathsf{ealz}(s)\textrm{ appears only once in }s\},\\
\CMcal{P}_{3}&:=\{s\in\CMcal{I}^*\,\backslash \CMcal{P}_1: \mathsf{ealz}(s)\textrm{ appears more than once in }s\}.
\end{align*}
For each subset, we are going to determine a generating function formula, which contributes to a functional equation of $\P(t;u,x,z,w,v,q)$ in the proof of Theorem \ref{T:1}.
\begin{lemma}\label{L:p1}
The generating function for $\CMcal{P}_1$ is 
	\begin{align*}
	&\quad \sum_{s\in\CMcal{P}_1}t^{\vert s\vert}u^{\asc(s)}x^{\mathsf{dist}(s)}q^{\max(s)}v^{\rmin(s)}
	z^{\mathsf{czero}(s)}w^{\mathsf{czero}(s)-\mathsf{ealz}(s)}\\
	&=\frac{zqxuv^2t^2(q-zw(q-1)t)}{1-zwt}.
	\end{align*}
\end{lemma}
\begin{proof}
Every $s$ in $\P_1$ has the form $(0,0,\ldots,0,1)$, which implies that the LHS of the above equation equals 
\begin{align*}
&\frac{zqxuv^2t^2}{1-zt}+q(q-1)xuv^2zt^2=\frac{zqxuv^2t^2(q-zw(q-1)t)}{1-zwt}.
\end{align*}
\end{proof}
Let $\alpha_1(t;u,x,v,q):=[z]\P(t;u,x,z,w,v,q)$, that is, the generating function of inversion sequences from $\I^*$ (also $\I$) with only one zero. Then,
\begin{lemma}\label{Lzero:case2}
	The generating function for $\P_2$ is
	\begin{align*}
	\nonumber&\,\quad \sum_{s\in\P_2}t^{\vert s\vert}u^{\asc(s)}x^{\mathsf{dist}(s)}z^{\mathsf{czero}(s)}w^{\mathsf{czero}(s)-\mathsf{ealz}(s)}v^{\rmin(s)}q^{\max(s)}\\
	&=tx\left(q+\frac{w}{1-w}\right)\CMcal{P}(t;u,x,z,1,v,q)+\frac{tx(u-1)}{1-w}
	\CMcal{P}(t;u,x,z,w,v,q)\\
	&\qquad\qquad+tx(u-1)(q-1)\CMcal{P}(t;u,x,z,0,v,q)\\
	&\qquad\qquad+\left(\frac{tux(v-vw-1)}{w(1-w)}+\frac{zt^2uxv}{1-zwt}\right)
	\CMcal{P}(t;u,x,zw,1,v,q)\\
	&\qquad\qquad +tux(q-1)(v-1)\alpha_1(t;u,x,v,q).
	\end{align*}
\end{lemma}
\begin{proof}
	For every sequence $s\in\I^*\backslash\P_1$ with $\czero(s)=p$, $s$ has the form $(0,0,\ldots,0,i,j,\ldots)$ with $j\ge 1$ and $1\le i\le p$.  
	Now we divide the set $\P_2$ into three disjoint ones, according to the change of statistic $\asc$.
	\begin{align*}
	\P_{2,1}&:=\{s\in\P_2: \ealz(s)>s_{\czero(s)+2},\,\min\{s_i: i\ge \czero(s)+2\}=1\},\\
	\P_{2,2}&:=\{s\in\P_2: \ealz(s)<s_{\czero(s)+2},\,\min\{s_i: i\ge \czero(s)+2\}\le 2\},\\
	\P_{2,3}&:=\{s\in\P_2: \ealz(s)=1,\,\min\{s_i: i\ge \czero(s)+2\}\ge 3\}.
	\end{align*}
	We are going to treat each subset $\P_{2,k}$ separately. For the subset $\P_{2,1}$, there is a bijection 
	\begin{align*}
	f_1:\P_{2,1}\cap\I_n\rightarrow \{(s,j):s\in\I^*\cap\I_{n-1}\, \mbox{ and }\, j>\ealz(s)\}
	\end{align*}
	that sends $s$ to a pair $(s',j)$ where $s'$ is obtained from $s$ by removing $\ealz(s)$ and replacing any entry $y$ by $y-1$ whenever $y>\ealz(s)$. Take $j=\ealz(s)>\ealz(s')$. It is readily seen that 
	\begin{align*}
	(\asc,\zero,\rmin)s=(\asc,\zero,\rmin)s',
	\end{align*}
	$\dist(s)=\dist(s')+1$ and $|s|=|s'|+1$. In fact $f_1$ is a bijection and its inverse map $f_1^{-1}$ tells us how to construct the generating function for $\P_{2,1}$ from the one $\P(t;u,x,z,w,v,q)$ for $\I^*$, To be precise, we do the following substitutions (assume that $\czero(s)=p$ and $\ealz(s)=j$)
	\begin{align*}
	\quad w^{p-j}\rightarrow w^{p-j-1}+\cdots+w^{1}+qw^{0}&=\frac{1-w^{p-j}}{1-w}+(q-1),\\
	&=q+\frac{w}{1-w}-\frac{w^{p-j}}{1-w}\quad \text{for  $1\le j\le p-1$},
	\end{align*}
	and $w^0\rightarrow 0$ in $(tx)\CMcal{P}(t;u,x,z,w,v,q)$, obtaining the generating function for $\P_{2,1}$:
	\begin{align}\label{E:2part1}
	\nonumber&tx\left(q+\frac{w}{1-w}\right)\CMcal{P}(t;u,x,z,1,v,q)-\frac{tx}{1-w}
	\CMcal{P}(t;u,x,z,w,v,q)\\
	&\qquad-tx(q-1)\CMcal{P}(t;u,x,z,0,v,q).
	\end{align}
     Apply the same bijection $f_1$ on $\P_{2,2}$, that is,
     \begin{align*}
     f_1:\P_{2,2}\cap\I_n\rightarrow \{(s,j):s\in\I^*\cap \I_{n-1}\, \mbox{ and }\, 1\le j\le \ealz(s)\}
     \end{align*}
      sends $s$ to a pair $(s',j)$ with the only differences that $j=\ealz(s)\le \ealz(s')$ and $\asc(s)=\asc(s')+1$. Here we have to distinguish the situations when the statistics $\rmin$ and $\max$ are increased by one, respectively. In terms of generating functions, we do the substitution (assume that $\czero(s)=p$ and $\ealz(s)=j$) 
	\begin{align*}
	\quad w^{p-j}&\rightarrow vw^{p-1}+\cdots+w^{p-j}=\frac{w^{p-j}-w^{p}}{1-w}+(v-1)w^{p-1}\quad \text{for  $1\le j\le p-1$},\\
	w^0&\rightarrow vw^{p-1}+\cdots +w^1+qw^0=\frac{1-w^{p}}{1-w}+(v-1)w^{p-1}+q-1 \quad \text{ for  $j=p\ne 1$,}\\
	w^0&\rightarrow qv \quad \text{ for  $j=p=1$}
	\end{align*}
	in $tux\,\CMcal{P}(t;u,x,z,w,v,q)$, leading to the generating function for $\P_{2,2}$:
	\begin{align}\label{E:2part2}
	\nonumber&\frac{tux}{1-w}\CMcal{P}(t;u,x,z,w,v,q)+\frac{tux(v-vw-1)}{w(1-w)}
	\CMcal{P}(t;u,x,zw,1,v,q)\\
	&\qquad+tux(q-1)\CMcal{P}(t;u,x,z,0,v,q)+tux(q-1)(v-1)\alpha_1(t;u,x,v,q).
	\end{align}
    We next discuss the remaining subset $\P_{2,3}$. There is a bijection
	\begin{align}\label{E:f2}
	f_2:\P_{2,3}\cap\I_n\rightarrow \dot\bigcup_{k\ge 1}\{s\in\I_{n-1}:\czero(s)=\zero(s)\,\mbox{ and }\,\min(s\backslash \{0\})=k+1\}
	\end{align}
	that transforms $s\in\P_{2,3}$ to an inversion sequence $s'$ with smallest non-zero entry at least $2$. This is achieved by removing the unique entry $1$ and decrease every non-zero entry by $1$ in $s$. Furthermore, 
	$|s'|=|s|+1$ and the value of the statistics $\asc,\dist,\rmin$ are decreased by one. On the other hand, every sequence $s\in\I$ satisfying $\czero(s)=\zero(s)$ and $\min(s\backslash \{0\})=k+1$ can be reduced to a sequence from $\I^*$. That is,
	\begin{align*}
	f_3: \{s\in\I_n:\czero(s)=\zero(s)\,\mbox{ and }\,\min(s\backslash \{0\})=k+1\}
	\rightarrow \{(s,k):s\in\I^*\cap\I_{n-k}\} 
	\end{align*}
    is a bijection defined as follows: Given a sequence from the LHS, it must contain at least $k$ zeros, we remove $k$ zeros from the beginning of $s$ and reduce every non-zero entry by $k$. Let $s'$ be the resulting sequence, it is not hard to examine that $(s',k)$ belongs to the image set of $f_3$ and the value of statistics $\zero,\ealz$ are decreased by $k$. Using the bijection $f_3^{-1}$, we can express the generating function $\P^*$ for the set on the RHS of (\ref{E:f2}), namely, 
    \begin{align*}
    \P^*=\sum_{k=1}^{\infty}(zwt)^k\CMcal{P}(t;u,x,z,w,v,q).
    \end{align*}
     Then according to the bijection $f_2^{-1}$, we do the substitution $w^{p-j}\rightarrow w^{p-1}$ in the generating function $tuxv\,\P^*$, yielding the generating function for $\P_{2,3}$, i.e.,
     \begin{align}\label{E:2part3}
     \frac{tuxv}{w}\sum_{k=1}^{\infty}(zwt)^k\CMcal{P}(t;u,x,zw,1,v,q)=
     \frac{zt^2uxv}{1-zwt}\CMcal{P}(t;u,x,zw,1,v,q).
     \end{align}
     Adding the terms (\ref{E:2part1}), (\ref{E:2part2}), (\ref{E:2part3}), we establish the formula in Lemma \ref{Lzero:case2}.
\end{proof}
\begin{lemma}\label{Lzero:case3}
	The generating function for $\P_3$ is
	\begin{align}
	\nonumber&\,\quad \sum_{s\in\P_3}t^{\vert s\vert}u^{\asc(s)}x^{\mathsf{dist}(s)}z^{\mathsf{czero}(s)}w^{\mathsf{czero}(s)-\mathsf{ealz}(s)}v^{\rmin(s)}q^{\max(s)}
	\label{eqL:case3}\\
	\nonumber&=\frac{1-zt}{z}\left(q+\frac{w}{1-w}\right)\CMcal{P}(t;u,x,z,1,v,q)+\frac{(u-1)(1-zt)}{z(1-w)}\CMcal{P}(t;u,x,z,w,v,q)\\
	\nonumber&\qquad+(q-1)(u-1)\frac{1-zt}{z}\CMcal{P}(t;u,x,z,0,v,q)\\
	&\qquad-\frac{u(1-zt)}{zw(1-w)}\CMcal{P}(t;u,x,zw,1,v,q)
	-u(1-zt)(q-1)\alpha_1(t;u,x,q,v).
	\end{align}
\end{lemma}
\begin{proof}
	Every sequence $s\in \P_3$ with $\mathsf{czero}(s)=p-1$ and $\mathsf{ealz}(s)=i$ has the form
	$$
	s=(0,0,\ldots,0,i,j,s_{p+2},\ldots)
	$$
	where $1\le i\le p-1$, $1\le j\le p$ and the entry $i$ appears at least twice. Define a map
	\begin{align}\label{E:xi1}
	\xi_1:\P_3\cap\I_n\rightarrow \{(s,i):s\in \I^*\cap\I_n \,\mbox{ and }\, i\in s\,\mbox{ and } 1\le i<\czero(s)\},
	\end{align}
	by $\xi_1(s)=(s^*,\ealz(s))$ where $s^*=(0,0,\ldots,0,0,j,s_{p+2},\ldots)$ satisfying $\czero(s^*)=p$ and $0\ne \ealz(s)\in s^*$. In order to find the generating function for such sequences $s^*$, we count its complement. That is, for a given $i$ with $i\ge 2$, we derive the generating function for sequences $s\in \I^*$ such that $i\not\in s$. Define a bijection (for a fixed $i$)
	\begin{align}\label{E:xi11}
	\xi_i: \{s\in \I^*\cap\I_n: i\not \in s \mbox{ and }\,\czero(s)\ge i+1\}\rightarrow \{s\in \I^*\cap\I_{n-1}: \czero(s)\ge i\}
	\end{align}
	by $\xi_i(s)=s'$ where $s'$ is obtained from $s$ by removing the last zero and decrease every entry $y$ by one if $y>i$. The change of statistic $\ealz$ depends on the value of $\ealz(s')$, so we have to discuss the case when $\ealz(s)<i<\czero(s)$ occurs. For a given $i\ge 2$, let
	\begin{align*}
	\CMcal{S}_{1}(i):&=\{s\in \I^*: i\not \in s\,\mbox{ and }\, \ealz(s)<i<\czero(s)\},\\
	\CMcal{S}_2(i):&=\I^*-\CMcal{S}_1(i)
	=\{s\in\I^*:\ealz(s)=\czero(s)\, \mbox{ or }\,\czero(s)-1\}\\
	&\qquad\qquad\qquad\qquad\dot\cup \{s\in\I^*: i\in s\,\mbox{ and }\,\ealz(s)<i<\czero(s)\}.
	\end{align*}
	The bijection $\xi_i$ (see (\ref{E:xi11})) restricted to $\CMcal{S}_1(i)$ gives
	\begin{align}\label{E:xi21}
    \xi_i: S_1(i)\cap\I_n\rightarrow \{s\in \I^*\cap\I_{n-1}:\ealz(s)< i\le \czero(s)\}
	\end{align}
	such that the values of statistics $\asc,\dist,\rmin,\max,\ealz$ stay the same, while $\czero$ and $\czero-\ealz$ are decreased by one. Observe that only when $s$ satisfies $\ealz(s)\ne \czero(s)$, there is an $i$ such that $(s,i)$ belongs to the image set of (\ref{E:xi21}).
	Let us recall that the generating function for the sequences $s\in\I^*$ such that $\ealz(s)\ne \czero(s)$ is $\CMcal{P}(t;u,x,z,w,v,q)-\CMcal{P}(t;u,x,z,0,v,q)$, then it follows from (\ref{E:xi21}) that for any fixed $i$, the generating function for the set $\CMcal{S}_1(i)$ is	
	\begin{align}\label{E:genS1}
	ztw(\CMcal{P}(t;u,x,z,w,v,q)-\CMcal{P}(t;u,x,z,0,v,q)).
	\end{align}
	This indeed implies the generating function
	\begin{align}\label{E:genS2}
	(1-ztw)\CMcal{P}(t;u,x,z,w,v,q)+ztw\CMcal{P}(t;u,x,z,0,v,q),
	\end{align}
	 for its complement set $\CMcal{S}_2(i)$. In the next step, we will present the generating function for $\P_3$ through the bijection $\xi_1$ (see (\ref{E:xi1})). For this purpose, we partition
	the set $\P_3$ into two disjoint subsets:
	\begin{align*}
	\P_{3,1}:=\{s\in\P_3: \ealz(s)\ge s_{\czero(s)+2}\}\,\mbox{ and }\,
	\P_{3,2}:=\{s\in\P_3: \ealz(s)<s_{\czero(s)+2}\}.
	\end{align*}
	By (\ref{E:xi1}), we have $\xi_1:\P_{3,1}\cap\I_n\rightarrow \{(s,i):s\in\I^*\cap\I_n,\, i\in s \mbox{ and }\,\ealz(s)\le i<\czero(s)\}$. Observe that 
	\begin{align}\label{E:setdiv}
	\nonumber&\quad\{(s,i):s\in\I^*\cap\I_n, i\in s \mbox{ and }\,\ealz(s)\le i< \czero(s)\}\\
	\nonumber&=\{(s,i):s\in\I^*\cap\I_n, i\in s \mbox{ and }\,\ealz(s)<i< \czero(s)\}\\
	\nonumber&\quad\quad\dot\cup\{(s,\ealz(s)):s\in\I^*\cap\I_n\,\mbox{ and }\,\ealz(s)\ne \czero(s)\}\\
	&=\{(s,i):s\in\CMcal{S}_2(i)\cap\I_n,\ealz(s)\le i< \czero(s)\}\\
	\nonumber&\quad\quad\dot\cup\{(s,\ealz(s)):s\in \CMcal{S}_1(i)\cap\I_n\},
	\end{align}
    we are going to do the substitutions, respectively, in the generating functions for the two subsets in (\ref{E:setdiv}),
    so as to deduce the generating function for the set $\P_{3,1}$. That is,   
    substitute (assume that $\czero(s)=p$ and $\ealz(s)=j$)
	\begin{align*}
	\quad w^{p-j}\rightarrow w^{p-1-j}+\cdots+w^{1}+qw^{0}&=\frac{1-w^{p-j-1}}{1-w}+(q-1)\quad \text{for  $1\le j\le p-1$}\\
	&=\frac{q-qw+w}{1-w}-\frac{w^{p-j}}{1-w},
	\end{align*}
	and $w^0\rightarrow 0$ in $z^{-1}\times (\ref{E:genS2})$. 
	This yields the generating function for sequences $s$ in $\P_{3,1}$ whose image $\xi_1(s)$ belongs to the first set of (\ref{E:setdiv}). Furthermore, dividing (\ref{E:genS1}) by $zw$ gives the generating function for sequences $s$ in $\P_{3,1}$ whose image $\xi_1(s)$ belongs to the second set of (\ref{E:setdiv}). Summing these two parts equals
	\begin{align}\label{E:p31}
	\nonumber&\quad\sum_{s\in\P_{3,1}}t^{\vert s\vert}x^{\rep(s)}q^{\fmax(s)}w^{\ealm(s)}u^{\asc(s)}z^{\zero(s)}\\
	\nonumber&=\frac{1-zt}{z}\left(q+\frac{w}{1-w}\right)\CMcal{P}(t;u,x,z,1,v,q)-\frac{1-zt}{z(1-w)}\CMcal{P}(t;u,x,z,w,v,q)\\
	&\qquad-(q-1)(t-z^{-1})\CMcal{P}(t;u,x,z,0,v,q).
	\end{align}
	We now turn to establish the generating function for the set $\P_{3,2}$. For a given $i\ge 2$, let 
	\begin{align*}
	\CMcal{T}_{1}(i):&=\{s\in \I^*: i\not \in s\,\mbox{ and }\, 2\le i<\ealz(s)\},\\
	\CMcal{T}_2(i):&=\{s\in\I^*:\ealz(s)\ne 1\}-\CMcal{T}_1(i)\\
	&=\{s\in\I^*:\ealz(s)=2\}
	\dot\cup \{s\in\I^*: i\in s\,\mbox{ and }\,2\le i<\ealz(s)\}.
	\end{align*}
	Then by (\ref{E:xi11}), we have 
	\begin{align}\label{E:xi3}
	\xi_i:\T_1(i)\cap\I_n\rightarrow \{s\in\I^*\cap\I_n: 2\le i\le\ealz(s)\}
	\end{align}
	such that the values of statistics $\asc,\dist,\rmin,\max,(\czero-\ealz)$ stay the same, while $\czero$ is decreased by one. Only when $\ealz(s)=1$, there is no $i$ satisfying (\ref{E:xi3}) and therefore no such bijection $\xi_i$, implying that the generating function for the image set of (\ref{E:xi3}) is $\CMcal{P}(t;u,x,z,w,v,q)-\gamma_1(t;u,x,z,w,v,q)$, where $\gamma_1$ counts the sequences $s$ with $\mathsf{ealz}(s)=1$.	
	As a result of $\xi_i$, the generating function for the subset $\T_1(i)$ is
	\begin{align*}
	zt\,(\CMcal{P}(t;u,x,z,w,v,q)-\gamma_1(t;u,x,z,w,v,q)).
	\end{align*}	
	In other words, the generating function for the set $\T_2(i)$ is
	\begin{align*}
	(1-zt)\,(\CMcal{P}(t;u,x,z,w,v,q)-\gamma_1(t;u,x,z,w,v,q)).
	\end{align*} 
	Using (\ref{E:xi1}), we find that $\xi_1:\P_{3,2}\cap\I_n\rightarrow \{(s,i):s\in \I^*\cap\I_n,i\in s\mbox{ and } 1\le i<\ealz(s)\}$. Since the image set equals
	\begin{align}\label{E:setdiv2}
	\nonumber&\quad\{(s,i):s\in \I^*\cap\I_n,\, i\in s\mbox{ and } 1\le i<\ealz(s)\}\\
	\nonumber&=\{(s,i):s\in\I^*\cap\I_n: i\in s\,\mbox{ and }\,2\le i<\ealz(s)\}
	\dot\cup \{(s,1):s\in\I^*\,\mbox{ and }\,\ealz(s)\ge 2\}\\
	&=\{(s,i):s\in\T_2(i): i\in s\,\mbox{ and }\,1\le i<\ealz(s)\}
	\dot\cup \{(s,1):s\in\T_1(i)\},
	\end{align}
	we do the substitutions below, respectively, in the generating functions for the two sets in (\ref{E:setdiv2}) in order to establish the generating function for the set $\P_{3,2}$. Replace (assume that $\czero(s)=p$ and $\ealz(s)=j$)
	\begin{align*}
	w^{p-j}&\rightarrow w^{p-j}+w^{p-j+1}+\cdots+w^{p-2}=\frac{w^{p-j}-w^{p-1}}{1-w} \quad\mbox{ for } j<p,\\
	w^{0}&\rightarrow qw^0+w^{1}+\cdots+w^{p-2}=\frac{1-w^{p-1}}{1-w}+q-1 \quad\mbox{ for } j=p
	\end{align*}
	in the generating function $z^{-1}u(1-zt)\,(\CMcal{P}(t;u,x,z,w,v,q)-\gamma_1(t;u,x,z,w,v,q))$; substitute $w^{p-j}\rightarrow w^{p-2}$ in $ut(\CMcal{P}(t;u,x,z,w,v,q)-\gamma_1(t;u,x,z,w,v,q))$. This yields that 
	\begin{align}\label{E:p32}
	\nonumber&\quad\sum_{s\in\P_{3,2}}t^{\vert s\vert}x^{\rep(s)}q^{\fmax(s)}w^{\ealm(s)}u^{\asc(s)}z^{\zero(s)}\\
	\nonumber&=\frac{u(1-zt)}{z(1-w)}(\CMcal{P}(t;u,x,z,w,v,q)-\gamma_1(t;u,x,z,w,v,q))\\
	\nonumber&\qquad-\frac{u(1-zt)}{zw(1-w)}(\CMcal{P}(t;u,x,zw,1,v,q)-\gamma_1(t;u,x,zw,1,v,q))\\
	\nonumber&\qquad+\frac{u(1-zt)}{z}(q-1)(\CMcal{P}(t;u,x,z,0,v,q)-z\alpha_1(t;u,x,q,v))\\
	\nonumber&=\frac{u(1-zt)}{z(1-w)}\CMcal{P}(t;u,x,z,w,v,q)-\frac{u(1-zt)}{zw(1-w)}\CMcal{P}(t;u,x,zw,1,v,q)\\
	&\qquad+\frac{u(1-zt)}{z}(q-1)\CMcal{P}(t;u,x,z,0,v,q)-u(1-zt)(q-1)\alpha_1(t;u,x,q,v).
	\end{align}
	 The second equation holds because by definition $\gamma_1(t;u,x,z,0,v,q)=z\alpha_1(t;u,x,q,v)$, and  $\gamma_1(t;u,x,z,w,v,q)=w^{-1}\gamma_1(t;u,x,zw,1,v,q)$. Summing (\ref{E:p31}) and (\ref{E:p32}), we conclude (\ref{eqL:case3}), which finishes the proof.
\end{proof}
Using Lemma \ref{L:p1}--\ref{Lzero:case3} and the kernel method, we are now able to prove Theorem \ref{T:1}.

{\em Proof of Theorem \ref{T:1}}. Since the sum of the generating functions for the sets $\P_i$, $1\le i\le 3$, equals $\P(t;u,x,z,w,v,q)$, we have, with $\kappa:=tx-t+z^{-1}$, that
\begin{align}
\nonumber&\quad\left(1-\frac{\kappa(u-1)}{1-w}\right)\CMcal{P}(t;u,x,z,w,v,q)\\
\nonumber&=\frac{zqxuv^2t^2(q-zwt(q-1))}{1-zwt}
+\kappa\left(q-1+\frac{1}{1-w}\right)\CMcal{P}(t;u,x,z,1,v,q)\\
\nonumber&\qquad+\frac{tuxv(1-w)-u\kappa(1-zwt)}{w(1-zwt)(1-w)}
\CMcal{P}(t;u,x,zw,1,v,q)\\
\nonumber&\qquad+(q-1)(u-1)\kappa\CMcal{P}(t;u,x,z,0,v,q)\\
\label{E:funp1}&\qquad+uz(q-1)(txv-\kappa)\alpha_1(t;u,x,q,v).
\end{align}
Setting $w=0$ on both sides produces
\begin{align}
\nonumber\left(1-q^{-1}\right)\CMcal{P}(t;u,x,z,0,v,q)&=zq(q-1)xuv^2t^2+\kappa(q-1)\CMcal{P}(t;u,x,z,1,v,q)\\
\nonumber&\qquad+(q-1)\kappa(u-1)\CMcal{P}(t;u,x,z,0,v,q)\\
\label{E:funcp2}&\qquad+uz(q-1)(txv-\kappa)\alpha_1(t;u,x,q,v).
\end{align}
Now (\ref{E:funp1}) minus (\ref{E:funcp2}) is equal to
\begin{align*}
\left(1-\frac{\kappa(u-1)}{1-w}\right)\CMcal{P}(t;u,x,z,w,v,q)&=\frac{zqxuv^2t^2}{1-zwt}+\frac{\kappa}{1-w}\CMcal{P}(t;u,x,z,1,v,q)\\
&\qquad+\left(1-q^{-1}\right)\CMcal{P}(t;u,x,z,0,v,q).\\
&\qquad+\frac{tuxv(1-w)-u\kappa(1-zwt)}{w(1-zwt)(1-w)}
\CMcal{P}(t;u,x,zw,1,v,q).
\end{align*}
Choose $w=1-\kappa(u-1)$, so that the LHS of the above equation vanishes. This is the starting point of the kernel method. In consequence, the above equation can be further simplied into 
\begin{align*}
\CMcal{P}(t;u,x,zw,1,v,q)&=\frac{zqxv^2t^2w(u-1)}{1-zwt-txv(u-1)}
+\frac{u^{-1}w(1-zwt)}{1-zwt-txv(u-1)}\CMcal{P}(t;u,x,z,1,v,q)\\
&\qquad+\frac{(1-q^{-1})u^{-1}w(1-zwt)(u-1)}{1-zwt-txv(u-1)}\CMcal{P}(t;u,x,z,0,v,q).
\end{align*}
Replacing $\CMcal{P}(t;u,x,z,0,v,q)$ by the following expression found from (\ref{E:funcp2}),
\begin{align*}
\CMcal{P}(t;u,x,z,0,v,q)&=\frac{zq^2xuv^2t^2}{1-q\kappa(u-1)}
+\frac{q\kappa}{1-q\kappa(u-1)}\CMcal{P}(t;u,x,z,1,v,q)\\
&\qquad+\frac{uqz(txv-\kappa)}{1-q\kappa(u-1)}\alpha_1(t;u,x,q,v),
\end{align*}
we are led to
\begin{align}
\nonumber\frac{1}{z}\CMcal{P}(t;u,x,z,1,v,q)
&=\frac{uqxv^2t^2(1-u)}{1-zwt}\left(q+(1-q)zt\right)\\
\nonumber&\qquad+\frac{u(1-zwt-txv(u-1))(1-q+qw)}{w(1-zwt)}\frac{1}{zw}\CMcal{P}(t;u,x,zw,1,v,q)\\
\label{E:ker0}&\qquad+\frac{u(q-1)}{w}\left(1-w-txv(u-1)\right)\alpha_1(t;u,x,v,q),
\end{align}
Define $\gamma=1-t(x-1)(u-1)$ and $t(x-1)\sigma_k=(zt(x-1)+1)\gamma^k-1$. In particular, 
$\sigma_0=z$ and $\sigma_1=zw=z-z\kappa(u-1)$. We next iterate (\ref{E:ker0}) (the second part of the kernel method), obtaining that 
\begin{align*}
\frac{1}{z}\CMcal{P}(t;u,x,z,1,v,q)&=uqxv^2t^2(1-u)
\sum_{n=1}^{m}\frac{\left(q+(1-q)\sigma_{n-1}t\right)}{1-\sigma_nt}\CMcal{Q}_{n-1}(z)\\
&\qquad+u(q-1)\alpha_1(t;u,x,v,q)\sum_{n=1}^{m}\frac{\CMcal{Q}_{n-1}(z)}{\sigma_n}\left((1-txv(u-1))\sigma_{n-1}-\sigma_n\right)\\
&\qquad+\frac{\CMcal{Q}_m(z)}{\sigma_m}\P(t,u,x,\sigma_m,1,v,q),
\end{align*}
where $\CMcal{Q}_{n}(z)$ is a finite product:
\begin{align*}
\CMcal{Q}_{n}(z):=u^n\prod_{i=1}^{n}\frac{(1-\sigma_i t-txv(u-1))(\sigma_{i-1}(1-q)+q\sigma_i)}{\sigma_i(1-\sigma_i t)}.
\end{align*}
Letting $m\rightarrow\infty$, then
\begin{align*}
\frac{1}{z}\CMcal{P}(t;u,x,z,1,v,q)&=uqxv^2t^2(1-u)
\sum_{n=1}^{\infty}\frac{\left(q+(1-q)\sigma_{n-1}t\right)}{1-\sigma_nt}\CMcal{Q}_{n-1}(z)\\
&\qquad+u(q-1)\alpha_1(t;u,x,v,q)\sum_{n=1}^{\infty}\frac{\CMcal{Q}_{n-1}(z)}{\sigma_n}\left((1-txv(u-1))\sigma_{n-1}-\sigma_n\right).
\end{align*}
Take $z=0$ on both sides, we obtain
\begin{align*}
\alpha_1(t;u,x,v,q)&=uqxv^2t^2(1-u)
\sum_{n=1}^{\infty}\frac{qx-1+(1-q)\gamma^{n-1}}{x-\gamma^n}\CMcal{Q}_{n-1}(0)\\
&\quad+u(q-1)\alpha_1(t;u,x,v,q)\sum_{n=1}^{\infty}
\frac{(1-txv(u-1))(\gamma^{n-1}-1)-\gamma^n+1}{\gamma^n-1}\CMcal{Q}_{n-1}(0).
\end{align*}
It follows by the relation $\CMcal{G}(t(u-1)(x-1);x,u,v,q,1)=(qtuxv)^{-1}\alpha_1(t;u,x,v,q)$ that 
\begin{align*}
&\quad\CMcal{G}(t(u-1)(x-1);x,u,v,q,1)\\
&=vt(1-u)
\sum_{n=1}^{\infty}\frac{qx-1+(1-q)\gamma^{n-1}}{x-\gamma^n}\CMcal{Q}_{n-1}(0)\\
&\qquad\times\bigg(1-u(q-1)\sum_{n=1}^{\infty}
\frac{(1-txv(u-1))(\gamma^{n-1}-1)-\gamma^n+1}{\gamma^n-1}\CMcal{Q}_{n-1}(0)\bigg)^{-1}.
\end{align*}
This is equivalent to (\ref{E:thm1}). It remains to establish the (bi)-symmetric property of $\CMcal{G}(t;x,u,v,q,1)$, which is easily seen in the context of permutations. For any $\pi=\pi_1\pi_2\ldots\pi_n\in \S_n$,  
\begin{align*}
(\des,\ides,\lmin,\rmax)\pi&=(\des,\ides,\rmax,\lmin)(\pi^c)^r,\\
(\des,\ides,\lmin,\rmax)\pi&=(\ides,\des,\lmin,\rmax)\pi^{-1}
\end{align*}
where $\pi^c=(n+1-\pi_1)(n+1-\pi_2)\ldots(n+1-\pi_n)$ and $\pi^r=\pi_n\cdots\pi_2\pi_1$ are complement and reverse of $\pi$, respectively. Using the bijection $\Phi$ from (\ref{E:bv}), it is not hard to verify that 
\begin{align*}
(\asc,\dist,\max,\rmin)\Phi(\pi)
&=(\asc,\dist,\rmin,\max)\Phi((\pi^c)^r),\\
(\asc,\dist,\max,\rmin)\Phi(\pi)&=(\dist,\asc,\max,\rmin)\Phi(\pi^{-1}).
\end{align*}
This means $\CMcal{G}(t;x,u,v,q,1)=\CMcal{G}(t;x,u,q,v,1)=\CMcal{G}(t;u,x,q,v,1)$, which completes the proof of Theorem \ref{T:1}. \qed

We next prove Corollary \ref{C:1} via the formula of $\CMcal{G}(t;x,u,v,q,1)$ in (\ref{E:thm1}) and transformations on formal power series.

{\em Proof of Corollary \ref{C:1}}. 
Take $v=1$ in (\ref{E:adr2}), we obtain, with $\gamma=1-t$,
\begin{align*}
\frac{\CMcal{G}(t;x,u,1,1,1)}{(1-x)(1-u)}&=\sum_{n=1}^{\infty}\frac{tr^{n-1}(u-1)^{-1}}
{u^n(1-xr^{n-1})(1-xr^{n})}\\
&=\sum_{n=1}^{\infty}\left(\frac{1}{1-xr^n}-\frac{1}{1-xr^{n-1}}\right)\frac{(1-u)^{-1}}{u^nx}.
\end{align*} 
Expand $(1-xr^n)^{-1}$ and then interchange two sums, the above RHS becomes
\begin{align}\label{E:equi1}
\frac{1}{1-u}\sum_{n=1}^{\infty}\frac{1}{xu^n}\left(\frac{1}{1-xr^n}-\frac{1}{1-xr^{n-1}}\right)
&=\sum_{m=1}^{\infty}x^{m-1}\frac{r^m-1}{(u-r^m)(1-u)}\\
\nonumber&=\sum_{m=1}^{\infty}x^{m-1}\left(\frac{1}{u-1}-\frac{1}{u-r^m}\right)\\
\nonumber&=\frac{1}{(u-1)(1-x)}-\frac{1}{x}\sum_{m=1}^{\infty}\frac{x^m}{u-r^m},
\end{align}
where the last sum equals
\begin{align*}
-\frac{1}{x}\sum_{m=1}^{\infty}\frac{x^m}{u-r^m}&=\frac{1}{ux}\sum_{m=1}^{\infty}x^{m}\frac{u}{r^m-u}=\frac{1}{ux}\sum_{m=1}^{\infty}x^m\sum_{k=1}^{\infty}\frac{u^k}{r^{km}},
\end{align*}
by which Corollary \ref{C:1} follows. In a similar way one can show that the formula of $\CMcal{G}(t;x,u,1,1,1)$ in (\ref{E:adr1}) also equals (\ref{E:equi1}), i.e., (\ref{E:adr1}) and (\ref{E:adr2}) when $v=1$ are equivalent to the RHS of (\ref{E:gg1}). This finishes the proof.
\qed

\section{Proof of Theorem \ref{T:adr} and \ref{T:4}}\label{S:4}
Our proofs of Theorem \ref{T:adr} and \ref{T:4} are centered around a decomposition of the dual set of $\I^*$, namely, a set of inversion sequences with all maximal entries appearing at the beginning. 
The dual set of $\I^*$ is in simple bijection with $\I^*$, i.e., $s\mapsto s^c$ 
(where $s^c$ is the complement of $s$) that interchanges the number of maximal entries and zeros. However, this bijection carries no information on the number of distinct entries ($\dist$), so the first two decompositions of inversion sequences are not equivalent regarding the distribution of the Eulerian statistics. 

\begin{definition}For a sequence $s=(s_1,\ldots,s_n)\in\I_n$, let $\cmax(s)$ be the length of the initial strictly increasing subsequence of $s$. Let $\cmax(s)=p<n$, the {\bf e}ntry {\bf a}fter the {\bf l}ast {\bf m}aximal of $s$ is denoted by $\ealm(s)$, that is, $\ealm(s)=s_{p+1}$. For example, if $s=(0,1,2,3,2,4)$, then $\ealm(s)=2$. For convenience, we set $\ealm(0,1,\ldots,n-1)=0$.
\end{definition}
\begin{remark}
The statistic $\ealm$ was introduced in \cite{fjlyz} and it turns out to be quite useful in the refined enumerations of ascent sequences, a subset of inversion sequences (see \cite{js}).
\end{remark}
Define
\begin{align*}
\CMcal{I}^{\bullet}=\{s\in \CMcal{I}:\cmax(s)=\max(s)\,\mbox{ and }\,\zero(s)\ge 2\},
\end{align*}
Let $\CMcal{M}(t;x,q,w,u,v,z)$ be the generating function for the set $\I^\bullet$, i.e., 
\begin{align*}
\CMcal{M}(t;x,q,w,u,v,z):&=\sum_{s\in\I^\bullet}t^{|s|}
u^{\asc(s)}x^{\dist(s)}q^{\cmax(s)}w^{\ealm(s)}
v^{\rmin(s)}z^{\mathsf{zero}(s)},
\end{align*}
we will establish a functional equation for $\M(t;x,q,w,u,z,v)$ by dividing the set $\CMcal{I}^\bullet$ into the following subsets:
\begin{align*}
\CMcal{M}_{1}&:=\{s\in\CMcal{I}^\bullet: |s|=\cmax(s)+1\},\\
\CMcal{M}_{2}&:=\CMcal{I}^\bullet-\CMcal{M}_1.
\end{align*}
In the follow-up lemmas, we provide the generating function for the subsets $\M_1,\M_2$ separately. Let $\beta_1(t;x,u,v,z):=[q]\CMcal{M}(t;x,q,w,u,v,z)$ be the generating function of sequences $s\in\I^\bullet$ with only one maximal entry.
\begin{lemma}\label{L:M1}
	The generating funcion for the set $\M_1$ is
	\begin{align*}
	&\quad\sum_{s\in\CMcal{M}_1}t^{\vert s\vert}u^{\asc(s)}x^{\dist(s)}q^{\cmax(s)}z^{\zero(s)}w^{\ealm(s)}v^{\rmin(s)}=\frac{qz^2vt^2}{1-qxtu}.
	\end{align*}
\end{lemma}
\begin{proof}
	Every $s\in\M_1$ with $\cmax(s)=p$ has a form $s=(0,1,\ldots,p-1,0)$. Consequently, the LHS of the above equation equals
	\begin{align*}
	%&\quad\sum_{s\in\CMcal{M}_1}t^{\vert s\vert}u^{\asc(s)}x^{\dist(s)}q^{\cmax(s)}z^{\zero(s)}w^{\ealm(s)}v^{\rmin(s)}\\
	&\sum_{p\ge 1}z^2v\,u^{p-1}x^{p-1}q^pt^{p+1}=\frac{qz^2vt^2}{1-qxtu}.
	\end{align*}
\end{proof}
\begin{lemma}\label{L:M2}
	The generating function for $\CMcal{M}_2$ is
	\begin{align*}
	&u\left(z+\frac{w}{1-w}\right)\rho\CMcal{M}(t;x,q,1,u,v,z)+\frac{\rho(1-u)}{1-w}\CMcal{M}(t;x,q,w,u,v,z)\\
	&\quad +\rho(z-1)(1-u)\CMcal{M}(t;x,q,0,u,v,z)
	-\frac{t-tx+(qwu)^{-1}}{1-w}\CMcal{M}(t;x,qw,1,u,v,z)\\
	&\quad \qquad-\frac{z-1}{u}\beta_1(t;x,u,v,z)
	+\frac{z^2qu^2vtx\rho}{1-qtux}\CMcal{M}(t;x,q,1,u,v,1),
	\end{align*}
	where $\rho=t-tx+(qu)^{-1}$.
\end{lemma}
\begin{proof}
	Let $\CMcal{R}$ be a subset of $\I^\bullet$ such that the last maximal entry appears only once, and let $\CMcal{R}^c=\I^\bullet\,\backslash\CMcal{R}$ be its complement. Apply a bijection $g$ on the set $\R$, i.e., $g:\R\cap \I_n\rightarrow \I^\bullet\cap\I_{n-1}$ where $g(s)$ is obtained from $s$ by removing the last maximal entry and reduce each entry $y$ by one if $y\ge \max(s)$. Under the bijection $g$, the length, the number of maximal entries and ascents are all decreased by one. This implies that the generating function for the set $\R$ is 
	\begin{align}\label{E:r1}
	qxtu\M(t;x,q,w,u,v,z),
	\end{align}
	or equivalently, the one for the complement $\R^c$ is
	\begin{align}\label{E:r2}
	(1-qxtu)\CMcal{M}(t;x,q,w,u,v,z).
	\end{align}
	We divide the set $\M_2$ into two subsets:
	\begin{align*}
	\M_{2,1}&:=\{s\in\M_2: \cmax(s)\,\mbox{ appears in } s\};\\
	\M_{2,2}&:=\M_2-\M_{2,1},
	\end{align*} 
    and will build up a connection between $\M_{2,1}\cap\I_n$ and $\R^c\cap\I_n$. The one between $\M_{2,2}\cap\I_n$ and $\R\cap\I_n$ follows in the same way. Every sequence $s\in \CMcal{I}_n\cap\CMcal{M}_{2,1}$ with $\cmax(s)=p$ and $\ealm(s)=i$ has the form
	$$
	s=(0,1,2,\ldots,p-1,i,j,s_{p+3},\ldots, s_{n})
	$$
	where $0\le i\le p-1,0\le j\le p$ and the subsequence $(j,s_{p+3},\ldots,s_n)$ contains $p$.
	Define  $h(s)=(s^*,i)$, where
	$$s^*:=(0,1,2,\ldots,p-1,p,j,s_{p+3},\ldots, s_{n})\in \I_n\cap\CMcal{R}^c.
	$$ It is easy to verify that $h$ is bijection and it holds also when $\M_{2,1}$ is replaced by $\M_{2,2}$ and $\R^c$ by $\R$. The inverse map $h^{-1}$ enables us to find out the generating function for $\M_2$ by substitutions on the generating functions for $\R$ and $\R^c$, respectively. Before we proceed by substitutions, the three cases on the pairs $(s^*,i)$ below need to be discussed in order to describe the change of statistics $\asc$ and $\rmin$.
	\begin{enumerate}
	\item $i<j$ and $\zero(s^*)\ne 1$;
	\item $i\ge j$ and $\zero(s^*)\ne 1$;
	\item $i=0$ and $\zero(s^*)=1$.
	\end{enumerate}
    For all cases, $\max(s^*)=\max(s)+1$. 
    The first two cases differ only by the change of statistic $\asc$, namely, $\asc(s)=\asc(s^*)$ for case $(1)$, while $\asc(s)=\asc(s^*)+1$ in case $(2)$. Interpret case $(1)$ in terms of generating functions, we substitute
    $$
	\quad w^j\rightarrow z+w+\cdots+w^{j-1}=\frac{1-w^j}{1-w}+(z-1)\quad \text{for  $1\le j\le p$}
	$$
	and $w^0\rightarrow 0$, respectively, 	
	in the generating function $q^{-1}(1-qtux)\CMcal{M}(t;x,q,w,u,v,z)$ and 
	$ut\CMcal{M}(t;x,q,w,u,v,z)$, leading to the generating function for the sequences $s$ from $\M_2$ subject to condition $(1)$:
	\begin{align}
	\nonumber&\left(z+\frac{w}{1-w}\right)(tu+q^{-1}-tux)\CMcal{M}(t;x,q,1,u,v,z)-\frac{tu+q^{-1}-tux}{1-w}\CMcal{M}(t;x,q,w,u,v,z)\\
	\label{E:3case1}&\quad -(z-1)(tu+q^{-1}-tux)\CMcal{M}(t;x,q,0,u,v,z).
	\end{align}
	For case $(2)$, we substitute 
	\begin{align*}
	w^j &\rightarrow w^{j}+w^{j+1}+\cdots+w^{p-1}=\frac{w^j-w^p}{1-w}\quad \text{for  $1\le j\le p$}\\
	w^0&\rightarrow z+w+w^2+\cdots+w^{p-1}=\frac{1-w^p}{1-w}+(z-1),
	\end{align*}
	respectively, in $(qu)^{-1}(1-qtux)\CMcal{M}(t;x,q,w,u,v,z)$ and in $t\CMcal{M}(t;x,q,w,u,v,z)$, yielding the generating function for the sequences $s$ from $\M_2$ subject to condition $(2)$:
	\begin{align}
	\nonumber&\frac{t-tx+(qu)^{-1}}{1-w}\CMcal{M}(t;x,q,w,u,v,z)-\frac{t-tx+(qwu)^{-1}}{1-w}\CMcal{M}(t;x,qw,1,u,v,z)\\
	\label{E:3case2}&\quad \qquad+(z-1)(t-tx+(qu)^{-1})\CMcal{M}(t;x,q,0,u,v,z)-\frac{z-1}{u}\beta_1(t;x,u,v,z).
	\end{align}
	It remains to consider the case $(3)$. Every inversion sequence $s$ with smallest entry $k$ ($k\ge 1$) after the last maximal one can be reduced to a sequence from $\I^\bullet$ by a simple transformation below: Let 
	\begin{align*}
	\CMcal{C}_k:=\{s\in\I_n:\min\{s_{m},\cmax(s)+1\le m\le|s|\}=k \,\mbox{ and }\, \cmax(s)=\max(s)\},
	\end{align*}
	then define $g_1:\CMcal{C}_k 
	\rightarrow \{(s,k):s\in\I^\bullet\cap\I_{n-k}\,\mbox{ and }\, k\ge 1\}$ where $g_1(s)$ is obtained by removing the first $k$ entries and replacing each other entry $y$ by $y-k$ in $s$.  This process together with (\ref{E:r1}) implies that the generating function for the set $\CMcal{C}_k\,\cap \R$ is $$z(qutvx)^k(qxtu)\M(t;x,q,1,u,v,1).$$ 
	Similarly, by $g_1$ and (\ref{E:r2}), the generating function for $\CMcal{C}_k\,\cap \R^c$ equals $$z(qutvx)^k(1-qxtu)\M(t;x,q,1,u,v,1).$$
	Under the inverse bijection $h^{-1}$, every pair $(s^*,0)$ with $s^*\in\CMcal{C}_k$ is mapped into a sequence $s$ with exactly two zeros such that $\ealm(s)=0$ and the smallest entry after the second zero is $k$. Furthermore, $\rmin(s)=\rmin(s^*)-(k-1)$,   $\max(s)=\max(s^*)-1$ and $\zero(s)=\zero(s^*)+1$. If $s^*\in \R$, then $\dist(s^*)=\dist(s)-1$; otherwise $\dist(s^*)=\dist(s)$. Henceforth the generating functions of such sequences $s$ equals
\begin{align*}
&\quad\left(\frac{z}{v^{k-1}}\right)\frac{z}{qx}(qutvx)^k(qxtu)\M(t;x,q,1,u,v,1)\\&\quad+\left(\frac{z}{v^{k-1}}\right)\frac{z}{q}(qutvx)^k(1-qxtu)\M(t;x,q,1,u,v,1)\\
&=\left(\frac{z}{v^{k-1}}\right)\frac{z}{q}(qutvx)^k(qtu+1-qxtu)\M(t;x,q,1,u,v,1).
\end{align*}
    Summing over all $k$ for $k\ge 1$, the generating function for sequences $s$ satisfying condition $(3)$ is given by
	\begin{align*}
	&\quad (qtu+1-qxtu)\sum_{k=1}^{\infty}\left(\frac{z}{v^{k-1}}\right)\frac{z}{q}(qutvx)^k
	\M(t;x,q,1,u,v,1)\\
	&=\frac{z^2uvtx(qtu+1-qxtu)}{1-qtux}\CMcal{M}(t;x,q,1,u,v,1).
	\end{align*}
   Combining with (\ref{E:3case1}) and (\ref{E:3case2}) for cases $(1)$ and $(2)$, we obtain the generating funtion in Lemma \ref{L:M2}, which completes the proof.
\end{proof}
We are now ready to prove Theorem \ref{T:adr}.

{\em Proof of Theorem \ref{T:adr}}. Since (\ref{E:adr1}) is a special case of Theorem \ref{T:1}, it suffices to prove (\ref{E:adr2}).
By Lemma \ref{L:M1} and \ref{L:M2}, we find a functional equation of $\CMcal{M}(t;x,q,w,u,z)$:
\begin{align}\label{E:keyM1}
\nonumber\left(1-\frac{\rho(1-u)}{1-w}\right)\CMcal{M}(t;x,q,w,u,v,z)
&=\frac{qz^2vt^2}{1-qtxu}+u\left(z+\frac{w}{1-w}\right)\rho\CMcal{M}(t;x,q,1,u,v,z)\\
\nonumber&\quad +\rho(z-1)(1-u)\CMcal{M}(t;x,q,0,u,v,z)\\
\nonumber&-\frac{t-tx+(qwu)^{-1}}{1-w}\CMcal{M}(t;x,qw,1,u,v,z)\\
&-\frac{z-1}{u}\beta_1(t;x,u,v,z)
+\frac{z^2qu^2vtx\rho}{1-qtux}\CMcal{M}(t;x,q,1,u,v,1).
\end{align}
We set $z=1$, then
\begin{align*}
\left(1-\frac{\rho(1-u)}{1-w}\right)&\CMcal{M}(t;x,q,w,u,v,1)
=\frac{qvt^2}{1-qtxu}\\
&+\left(\frac{1}{1-w}+\frac{qxvut}{1-qxtu}\right)u\rho\CMcal{M}(t;x,q,1,u,v,1)\\
&-\frac{t-tx+(qwu)^{-1}}{1-w}\CMcal{M}(t;x,qw,1,u,v,1).
\end{align*}
We will employ the kernel method to solve this functional equation.
Take $w=1-(1-u)\rho$ so that the LHS becomes zero, we therefore have, with $\CMcal{C}(t;x,q,u,v,z):=u\rho\CMcal{M}(t;x,q,1,u,v,z)$, that
\begin{align*}
\CMcal{C}(t;x,q,u,v,1)&=\frac{qvt^2(w-1)}{1+utx((v-1)q-qwv)}\\
&\qquad\quad+\frac{1-qtux}{u(1+utx((v-1)q-qwv))}\CMcal{C}(t;x,qw,u,v,1).
\end{align*}
Define $ut(1-x)\tau_n=(qut(1-x)+1)\gamma^n-1$ where $\gamma=1-t(u-1)(x-1)$. By iterating the above equation, we conclude that
\begin{align}\label{E:eqC}
\nonumber\CMcal{C}(t;x,q,u,v,1)&=\sum_{n=1}^{\infty}\frac{vt^2(u-1)(qtu(1-x)+1)\gamma^{n-1}}{u^{n}(1-utx\tau_{n-1})}\\
&\qquad\quad\times\prod_{i=1}^{n}\frac{1-utx\tau_{i-1}}{1+utx((v-1)\tau_{i-1}-v\tau_i)}.
\end{align}
Since $\beta_1(t;x,u,v,1)=[q^0]\C(t;x,q,u,v,1)=t\CMcal{G}(t(u-1)(x-1);x,u,v,1,1)$, we see that 
\begin{align*}
&\quad\,\CMcal{G}(t(u-1)(x-1);x,u,v,1,1)\\
&=\sum_{n=1}^{\infty}\frac{vt(u-1)(1-x)\gamma^{n-1}}{u^{n}(1-x\gamma^{n-1})}
\prod_{i=1}^{n}\frac{1-x\gamma^{i-1}}{1+x(v-1)\gamma^{i-1}-xv\gamma^i},
\end{align*}
which is equivalent to (\ref{E:adr2}), thus completing the proof. \qed

Now we return to (\ref{E:keyM1}) and establish Theorem \ref{T:4} using the formula (\ref{E:adr2}).

{\em Proof of Theorem \ref{T:4}}. Setting $w=0$ on both sides of (\ref{E:keyM1}) produces
\begin{align}
\nonumber\left(1-\rho(1-u)\right)\CMcal{M}(t;x,q,0,u,v,z)
&=\frac{qz^2vt^2}{1-qtxu}+uz\rho\CMcal{M}(t;x,q,1,u,v,z)\\
\nonumber&\quad +\rho(z-1)(1-u)\CMcal{M}(t;x,q,0,u,v,z)\\
\label{E:wzero}&-\frac{z}{u}\beta_1(t;x,u,v,z)
+\frac{z^2qu^2vtx\rho}{1-qtux}\CMcal{M}(t;x,q,1,u,v,1).
\end{align}
On the other hand, set $w=1-(1-u)\rho$ on both sides of (\ref{E:keyM1}), we have
\begin{align}
\nonumber\frac{t-tx+(qwu)^{-1}}{1-w}\CMcal{M}(t;x,qw,1,u,v,z)&=\frac{qz^2vt^2}{1-qtxu}+u\left((z-1)\rho+\frac{1}{1-u}\right)\CMcal{M}(t;x,q,1,u,v,z)\\
&\nonumber\quad +\rho(z-1)(1-u)\CMcal{M}(t;x,q,0,u,v,z)\\
&\nonumber\quad-\frac{z-1}{u}\beta_1(t;x,u,v,z)\\
\label{E:wker}&\quad+\frac{z^2qu^2vtx\rho}{1-qtux}\CMcal{M}(t;x,q,1,u,v,1).
\end{align}
The difference of (\ref{E:wzero}) and (\ref{E:wker}) gives
\begin{align*}
\CMcal{M}(t;x,q,0,u,v,z)&=\frac{t-tx+(qwu)^{-1}}{(1-w)w}\CMcal{M}(t;x,qw,1,u,v,z)\\
&\quad+\frac{u}{u-1}\CMcal{M}(t;x,q,1,u,v,z)
-\frac{1}{uw}\beta_1(t;x,u,v,z).
\end{align*}
Plugging the above expression of $\CMcal{M}(t;x,q,0,u,v,z)$ into (\ref{E:wker}), we are led to 
\begin{align*}
\rho\CMcal{M}(t;x,q,1,u,v,z)&=\frac{qz^2vt^2(w-1)}{u(1-qxut)}
+\frac{z^2quvtx(w-1)}{1-qtux}\rho\CMcal{M}(t;x,q,1,u,v,1)\\
&\quad+\frac{(z-1)(1-w)}{u^2w}\beta_1(t;x,u,v,z)\\
&\quad+\frac{(1-z)(1-w)+w}{uw}(t-tx+(qwu)^{-1})\CMcal{M}(t;x,qw,1,u,v,z).
\end{align*}
Repeating the above equation, we deduce that 
\begin{align*}
\CMcal{C}(t;x,q,u,v,z)&=\sum_{n=1}^{\infty}\bigg(\frac{z^2vt^2(\tau_{n}-\tau_{n-1})}{1-xut\tau_{n-1}}
+\frac{z^2uvtx(\tau_n-\tau_{n-1})}{1-xut\tau_{n-1}}\CMcal{C}(t;x,\tau_{n-1},u,v,1)\\
&\quad+\frac{(z-1)(\tau_{n-1}-\tau_n)}{u\,\tau_{n}}\beta_1(t;x,u,v,z)\bigg)\times \prod_{i=1}^{n-1}\frac{(1-z)(\tau_{i-1}-\tau_i)+\tau_{i}}{u\,\tau_{i}}.
\end{align*}
In view of $\beta_1(t;x,u,v,z)=[q^0]\C(t;x,q,u,v,z)=zt\CMcal{G}(t(u-1)(x-1);x,u,v,1,z)$, we take the coefficient of $q^0$ on both sides and obtain that 
\begin{align*}
\CMcal{G}(t(u-1)(x-1);x,u,v,1,z)&=\sum_{n=1}^{\infty}\frac{zv\gamma^{n-1}(u-1)}{u(1-xut\,\varepsilon_{n-1})}(t+ux\,\CMcal{C}(t;x,\varepsilon_{n-1},u,v,1))A_n\\
&\qquad\quad\times\left(1-\sum_{n=1}^{\infty}\frac{(z-1)\gamma^{n-1}(1-u)}{u^2\,\varepsilon_{n}}A_n\right)^{-1},
\end{align*}
where $ut(1-x)\varepsilon_n=\gamma^n-1$ and $A_n$ is a finite product: 
\begin{align*}
A_n=\prod_{i=1}^{n-1}\frac{(1-z)\gamma^{i-1}(1-u)+u\varepsilon_{i}}{u^2\,\varepsilon_{i}}.
\end{align*}
By substitution $t$ by $t(u-1)^{-1}(x-1)^{-1}$ and (\ref{E:eqC}), we find the formula in Theorem \ref{T:4}. 
\qed

As stated in Remark \ref{R:1}, we will show $\CMcal{G}(t;x,u,v,1,1)=\CMcal{G}(t;u,x,v,1,1)$ by a transformation formula of $_4\phi_3$ series. Let us recall some definitions.

For indeterminates $a$ and $q$ (the latter is referred to as the base),
and non-negative integer $k$,
the basic shifted factorial (or $q$-shifted factorial) is defined as
\begin{equation*}
(a_1,\ldots,a_m;q)_k:=(a_1;q)_k\cdots(a_m;q)_k,
\end{equation*}
where again $k$ is a non-negative integer.

An ${}_\alpha\phi_\beta$ basic hypergeometric series with $\alpha$ upper
parameters $a_1,\ldots,a_\alpha$, and $\beta$ lower parameter
$b_1,\ldots,b_\beta$, base $q$ and argument $z$
is defined as
\begin{equation}\label{eq:bhyp}
{}_\alpha\phi_\beta\!\left[\begin{matrix}a_1,\ldots,a_\alpha\\
b_1,\ldots,b_\beta\end{matrix};q,z\right]:=
\sum_{k=0}^\infty\frac{(a_1,\ldots,a_\alpha;q)_k}
{(q,b_1,\ldots,b_\beta;q)_k}\left((-1)^kq^{\binom k2}\right)^{1+\beta-\alpha}z^k.
\end{equation}
\begin{theorem}[\cite{js}]\label{thm:4phi3tf}
	Let $a,b,c,d,e,r$ be complex variables, $j$ be a non-negative integer.
	Then, assuming that none of the denominator factors in \eqref{tf43}
	have vanishing constant term in $r$, we have the following transformation
	of convergent power series in $a$ and $r$:
	\begin{align}\label{tf43}
	&{}_4\phi_3\!\left[\begin{matrix}(1-r)^j,1-a,b,c\\
	d,e,(1-r)^{j+1}(1-a)bc/de\end{matrix};1-r,1-r\right]\notag\\
	&=\frac{((1-r)/e,(1-r)(1-a)bc/de;1-r)_j}
	{((1-r)(1-a)/e,(1-r)bc/de;1-r)_j}\notag\\
	&\quad\;\times{}_4\phi_3\!\left[\begin{matrix}(1-r)^j,1-a,d/b,d/c\\
	d,de/bc,(1-r)^{j+1}(1-a)/e\end{matrix};1-r,1-r\right].
	\end{align}
\end{theorem}
We now specialize the identity in \eqref{tf43} as follows:
We take $j=1$, do the substitutions
$r\mapsto t$, and $a\mapsto at$,
and do the following further replacements.
\begin{alignat*}{2}
&b\mapsto \frac{(1-qt)(1-qv)}{x(1-vt)(1-v)},&\qquad
&c\mapsto \frac{(1-vt)(1-qv)}{x(1-qt)(1-q)},\\
&d\mapsto \frac{(1-qv)^2}{ux(1-v)(1-q)},&\qquad
&e\mapsto\frac{(1-t)^2}x.     
\end{alignat*}
The result is the following transformation formula
of infinite series (which is actually even a formal power series
in $t$) which we write out in explicit terms:
\begin{align*}
& \quad\sum_{k\ge 0}\frac{\left( \frac{(1-qt)(1-qv)}{x(1-vt)(1-v)},
	\frac{(1-vt)(1-qv)}{x(1-qt)(1-q)},1-at;\,1-t\right)_{_k}}
{\left( \frac{(1-qv)^2}{ux(1-v)(1-q)},\frac{(1-t)^2}x,(1-at)u;\,
	1-t\right)_{_k}} (1-t)^k\\
&=\frac{(1-t-x)\,(1-t-u+uat)}
{(1-t-u)\,(1-t-x+xat)}\\
&\quad\quad\times
\sum_{k\ge 0}\frac{\left( \frac{(1-qt)(1-qv)}{u(1-vt)(1-v)},
	\frac{(1-vt)(1-qv)}{u(1-qt)(1-q)},1-at;\,1-t\right)_{_k}}
{\left( \frac{(1-qv)^2}{ux(1-v)(1-q)},\frac{(1-t)^2}u,(1-at)x;\,
	1-t\right)_{_k}} (1-t)^k.
\end{align*}
As a consequence, we immediately deduce the following result:
\begin{corollary}\label{cor1}
	Let
	\begin{align*}
	&\widetilde{\CMcal{H}}_1(t;x,u,v,q,a)
	:=\frac{qvt(1-t-x+xat)}{(1-t)(1-t-x)}\\
	&\times \sum_{k\ge 0}\frac{\left( \frac{(1-qt)(1-qv)}{x(1-vt)(1-v)},
		\frac{(1-vt)(1-qv)}{x(1-qt)(1-q)},1-at;\,1-t\right)_{_k}}
	{\left( \frac{(1-qv)^2}{ux(1-v)(1-q)},\frac{(1-t)^2}x,(1-at)u;\,
		1-t\right)_{_k}} (1-t)^k.   
	\end{align*}
	Then we have the symmetries
	$$\widetilde{\CMcal{H}}_1(t;x,u,v,q,a)=
	\widetilde{\CMcal{H}}_1(t;u,x,v,q,a)=
	\widetilde{\CMcal{H}}_1(t;x,u,q,v,a)=
	\widetilde{\CMcal{H}}_1(t;u,x,q,v,a).$$
\end{corollary}
The special case $a=t^{-1}$ is of interest but then series
is not any more a formal power series in $t$:

\begin{corollary}\label{cor2}
	Let
	\begin{equation*}
	\CMcal{H}_1(t;x,u,v,q)
	:=\frac{qvt}{1-t-x}
	\sum_{k\ge 0}\frac{\left( \frac{(1-qt)(1-qv)}{x(1-vt)(1-v)},
		\frac{(1-vt)(1-qv)}{x(1-qt)(1-q)};\,1-t\right)_{_k}}
	{\left( \frac{(1-qv)^2}{ux(1-v)(1-q)},\frac{(1-t)^2}x;\,
		1-t\right)_{_k}} (1-t)^k.   
	\end{equation*}
	Then we have the symmetries
	$$\CMcal{H}_1(t;x,u,v,q)=\CMcal{H}_1(t;u,x,v,q)=
	\CMcal{H}_1(t;x,u,q,v)=\CMcal{H}_1(t;u,x,q,v).$$
\end{corollary}

A consequence of the $q\to 1$ limit of Corollary~\ref{cor2}
is the symmetry $\CMcal H(t;x,u,v,1)=\CMcal H(t;u,x,v,1)$,
which, in explicit terms, says that the series
\begin{equation*}
\CMcal{H}_1(t;x,u,v,1)
=\frac{vt}{1-t-x}
\sum_{k\ge 0}\frac{\left( \frac{(1-t)}{x(1-vt)};\,1-t\right)_{_k}}
{\left(\frac{(1-t)^2}x;\,
	1-t\right)_{_k}} u^k(1-vt)^k,
\end{equation*}
or equivalently (\ref{E:adr1}), is symmetric in $x$ and $u$. Next with a different substitution, we shall prove that (\ref{E:adr2}) is invariant if we interchange $u$ and $x$. 

We again choose $j=1$, do the substitutions
$r\mapsto t$, $a\mapsto at$ and 
\begin{alignat*}{2}
&b\mapsto \frac{x(1-qt)(1-qv)}{(1-t)(1-v)},&\qquad
&c\mapsto \frac{x(1-vt)(1-qv)}{(1-t)(1-q)},\\
&d\mapsto \frac{ux(1-qv)^2(1-qt)(1-vt)}{(1-t)^2(1-v)(1-q)},&\qquad
&e\mapsto x(1-vt)(1-t).     
\end{alignat*}
As a result, we arrive at
\begin{align*}
&\quad \sum_{k\ge 0}\frac{\left( \frac{x(1-qt)(1-qv)}{(1-t)(1-v)},
	\frac{x(1-vt)(1-qv)}{(1-t)(1-q)},1-at;\,1-t\right)_{_k}}
{\left( \frac{ux(1-qv)^2(1-qt)(1-vt)}{(1-t)^2(1-v)(1-q)},x(1-vt)(1-t),\frac{(1-at)(1-t)}{u(1-vt)};\,
	1-t\right)_{_k}} (1-t)^k\\
&=\frac{(x-xvt-1)\,(u-uvt-1+at)}
{(u-uvt-1)\,(x-xvt-1+at)}\\
&\quad\quad\times
\sum_{k\ge 0}\frac{\left( \frac{u(1-qt)(1-qv)}{(1-t)(1-v)},
	\frac{u(1-vt)(1-qv)}{(1-t)(1-q)},1-at;\,1-t\right)_{_k}}
{\left(\frac{ux(1-qv)^2(1-qt)(1-vt)}{(1-t)^2(1-v)(1-q)},u(1-vt)(1-t),\frac{(1-at)(1-t)}{x(1-vt)};\,
	1-t\right)_{_k}} (1-t)^k.
\end{align*}
It follows directly that 
\begin{corollary}
	Let 
	\begin{align*}
	\CMcal{H}_2^*&(t;x,u,q,v,a):=\frac{x-xvt-1+at}{x-xvt-1}\\
	&\times \sum_{k\ge 0}\frac{\left( \frac{x(1-qt)(1-qv)}{(1-t)(1-v)},
		\frac{x(1-vt)(1-qv)}{(1-t)(1-q)},1-at;\,1-t\right)_{_k}}
	{\left( \frac{ux(1-qv)^2(1-qt)(1-vt)}{(1-t)^2(1-v)(1-q)},x(1-vt)(1-t),\frac{(1-at)(1-t)}{u(1-vt)};\,
		1-t\right)_{_k}} (1-t)^k.
	\end{align*}
	Then we have the symmetries $	\CMcal{H}_2^*(t;x,u,q,v,a)=\CMcal{H}_2^*(t;u,x,q,v,a)$.
\end{corollary}
Choose $a=t^{-1}$ and let $q\rightarrow 1$, set
\begin{align*}
\CMcal{H}_2(t;x,u,1,v):=\frac{x(1-vt)}{x-xvt-1}
\sum_{k\ge 0}\frac{\left(x;\,1-t\right)_{_k}}
{\left(x(1-vt)(1-t);\,
	1-t\right)_{_k}} u^{-k}(1-t)^k,
\end{align*}
which satisfies $\CMcal{H}_2(t;x,u,1,v)=\CMcal{H}_2(t;u,x,1,v)$. Since (\ref{E:adr2})$=vt(ux(1-vt))^{-1}\CMcal{H}_2(t;x,u,1,v)$, (\ref{E:adr2}) is also symmetric in $u$ and $x$, as desired.

\section{Proof of Theorem \ref{T:asczeromax}}\label{S:5}
The purpose of this section is to present two applications of a typical construction of inversion sequences, namely, always add a new entry at the end. The first application (Theorem \ref{T:bij}) is a bijective proof of a quadruple equidistribution of Euler-Stirling statistics over $\I_n$, as a step towards solving a conjecture on a quintuple equidistribution (see Section \ref{S:final}). The second one is a short proof of Theorem \ref{T:asczeromax}.

\begin{definition}
Given any sequence $s=(s_1,s_2,\ldots,s_n)\in\CMcal{I}_n$, let
\begin{align*}
\mathsf{Rmin}(s):=\{s_i: s_i<s_j \mbox{ for all } j>i\}
\end{align*}
be a set-valued statistic. Let $\mathsf{last}(s)$ be the last entry of $s$, then $\mathsf{last}(s)$ is the largest integer of $\Rmin(s)$.
For instance, when $s=(0,0,2,1,4,3)$, $\Rmin(s)=\{0,1,3\}$.
\end{definition}

\begin{theorem}\label{T:bij}
The quadruple of statistics $(\asc,\zero, \max, \mathsf{Rmin})$ have the same distribution as the one $(\mathsf{dist}, \zero, \max, \mathsf{Rmin})$ on inversion sequences.
\end{theorem}
\begin{remark}
	Lin and Kim \cite{kl} discovered a bijection between the quadruples $(\asc,\max,\zero,\mathsf{last})$ and $(\dist,\max,\zero,\mathsf{last})$ of statistics over $\I_n$, but it seems that one can not prove Theorem \ref{T:bij}
    by their bijection.
\end{remark}
\begin{proof}
We shall prove it by induction. It is easy to check the trivial case for $n=1$. Now suppose that the claim is true for inversion sequences of length $n-1$, that is, for any set $A$ with $\max(A)=m$ (here $\max(A)$ means the largest element of $A$), we have 
\begin{align}\label{E:1}
\sum_{\substack{s\in\CMcal{I}_{n-1}\\\Rmin(s)=A}}u^{\asc(s)}q^{\max(s)}z^{\zero(s)}=\sum_{\substack{s\in\CMcal{I}_{n-1}\\\Rmin(s)=A}}u^{\mathsf{dist}(s)}q^{\max(s)}z^{\zero(s)}.
\end{align}
We will prove that (\ref{E:1}) is also true if $n-1$ is replaced by $n$. For any $j\ne n-1$, we have, by (\ref{E:1}), that for any $A$ with $\max(A)=j$ and $A':=A-\{j\}$,
\begin{align}
\nonumber\sum_{\substack{s\in\CMcal{I}_{n}\\ \Rmin(s)=A\\ s_{n-1}<s_n=j}}u^{\asc(s)}q^{\max(s)}z^{\zero(s)}
&=\sum_{\substack{s\in\CMcal{I}_{n-1}\\\Rmin(s)=A'}}u^{\asc(s)+1}q^{\max(s)}z^{\zero(s)}\\
\label{E:31}&=\sum_{\substack{s\in\CMcal{I}_{n-1}\\\Rmin(s)=A'}}u^{\mathsf{dist}(s)+1}q^{\max(s)}z^{\zero(s)}.
\end{align}
Given a sequence $s=(s_1,s_2,\ldots,s_{n-1})\in\I_{n-1}$ with $\Rmin(s)=A'$ and $\max(A')<j$, we will construct a new sequence $s'$ from $s$ by implementing the following steps:
\begin{itemize}
	\item move all entries $s_{j+1},s_{j+2},\ldots,s_{n-1}$ one position to the right;
	\item replace $s_m$ by $s_m+1$ whenever $s_m\ge j$ (so that the sequence contains no entry $j$);
	\item let $i_1<i_2<\cdots<i_k$ be a sequence of positions after $j+1$ at which the entries are less than $j$, then move the entry $s_{i_1}$ to $(j+1)$-th position; move the entry $s_{i_\ell}$ to $i_{\ell-1}$-th position for $2\le \ell\le k$ (this preserves the values of right-to-left-minima);
	\item add $j$ at the end (by which $j$ is a newly added right-to-left minimum).
\end{itemize}
Let $s'$ be the resulting sequence, it is straightforward to examine the change of statistics. The number of maximal entries and zeros stay the same as $j\ne n-1, \ne 0$, while the number of distinct entries is increased by one and $\Rmin(s')=A$. The above steps are reversible, so the map $s\mapsto s'$ is a bijection. This implies that
\begin{align*}
\sum_{\substack{s\in\CMcal{I}_{n-1}\\\Rmin(s)=A'}}u^{\mathsf{dist}(s)+1}q^{\max(s)}z^{\zero(s)}%\\
&=\sum_{\substack{s\in\CMcal{I}_{n}\\ \Rmin(s)=A\\ 
		s_n=j\in \mathsf{Dist}(s)}}u^{\mathsf{dist}(s)}q^{\max(s)}z^{\zero(s)}.
\end{align*}
Here $j\in\mathsf{Dist}(s)$ means that $j$ is a distinct entry of $s$. In combination of (\ref{E:31}), we see that 
\begin{align*}
\sum_{\substack{s\in\CMcal{I}_{n}\\ \Rmin(s)=A\\ s_{n-1}<s_n=j}}u^{\asc(s)}q^{\max(s)}z^{\zero(s)}
&=\sum_{\substack{s\in\CMcal{I}_{n}\\ \Rmin(s)=A\\ 
		s_n=j\in \mathsf{Dist}(s)}}u^{\mathsf{dist}(s)}q^{\max(s)}z^{\zero(s)}.
\end{align*}
Furthermore, it is easily seen that the above identity holds for $j=n-1$. It remains to show
\begin{align}\label{E:32}
\sum_{\substack{s\in\CMcal{I}_{n}\\ \Rmin(s)=A\\ s_{n-1}\ge s_n=j}}u^{\asc(s)}q^{\max(s)}z^{\zero(s)}
&=\sum_{\substack{s\in\CMcal{I}_{n}\\ \Rmin(s)=A\\ 
		s_n=j\not\in \mathsf{Dist}(s)}}u^{\mathsf{dist}(s)}q^{\max(s)}z^{\zero(s)}.
\end{align}
Since $j=0$ is a trivial case for this identity, we only prove it when $1\le j\le n-2$. For a given $j$, let $A_j'$ be any subset of $[n-1]$ satisfying $\{a\in A'_j:a\le j\}\cup\{j\}=A$ and $\max(A'_j)\ge j$, then by (\ref{E:1}),
\begin{align}
\sum_{\substack{s\in\CMcal{I}_{n}\\ \Rmin(s)=A\\ s_{n-1}\ge s_n=j}}u^{\asc(s)}q^{\max(s)}z^{\zero(s)}
\nonumber&=\sum_{A_j'\subseteq [n-1]}\sum_{\substack{s\in\CMcal{I}_{n-1}\\\Rmin(s)=A'_j}}u^{\asc(s)}q^{\max(s)}z^{\zero(s)}\\
\label{E:33}&=\sum_{A'_j\subseteq [n-1]}\sum_{\substack{s\in\CMcal{I}_{n-1}\\\Rmin(s)=A'_j}}u^{\dist(s)}q^{\max(s)}z^{\zero(s)}.
\end{align}
We will establish a bijection that maps $s\in\I_{n-1}$ with $\Rmin(s)=A'_j$ to a sequence $s'\in\I_n$ such that $\Rmin(s')=A$ and $s'_n=j\not\in\mathsf{Dist}(s')$, by which the RHS of (\ref{E:33}) equals the RHS of (\ref{E:32}). This will lead to (\ref{E:32}), thus completing the proof.

For a sequence $s=(s_1,s_2,\ldots,s_{n-1})\in\I_{n-1}$ with $\Rmin(s)=A'_j$, the corresponding $s'$ is defined case by case:
\begin{enumerate}
\item if $s_{n-1}=j$, then choose $s'=s_1\ldots s_{n-2}s_{n-1}j$, i.e., add a $j$ at the end of $s$. Since $\Rmin(s)=A_j'=A$, $\Rmin(s')=A$ and $s'_n=j\not\in\mathsf{Dist}(s')$;
\item if $s_{n-1}\ge j+1$, let $m=\ell-1-s_{\ell}$ where $\ell=\min\{r:s_{r}=s_{n-1}\}$ and let $k$ be the largest index for which $s_i=s_{n-1}$ for all $n-k\le i\le n-1$. Then, 
\begin{itemize}
\item move all entries $s_{j+m+1},s_{j+m+2},\ldots,s_{n-1}$ $k$ positions to the right and increase the entry $y$ by $k$ whenever $y\ge j+m$ (in order to preserve the number of maximal entries); 
\item put $j$ to the $(j+m+1)$-th position and add $(k-1)$ $j's$ after it; 
\item remove $(k-1)$ identical entries $s_{n-1}+k$ from the end and replace all other entries $s_{n-1}+k$ by $j$ (by which all entries $s_{n-1}+k$ are gone and the last entry must be $j$);
\end{itemize}
Let $s'$ be the resulting sequence, then by construction $\Rmin(s')=A$ and the last entry $j\not\in\mathsf{Dist}(s')$. 
\end{enumerate}
This map is in fact invertible. For any $s'\in\I_n$ with $\Rmin(s')=A$, if the last two elements equal $j$, then $s'$ comes from case $(1)$. Otherwise, if the leftmost $j$ is located at the $(j+m+1)$-th position and followed by $(k-1)$ identical $j$'s with $m\ge 0$ and $k\ge1$, then we reverse the steps in (2) and recover the sequence $s$. It follows that $s\mapsto s'$ is a bijection and consequently (\ref{E:32}) is true. This complete the induction and thereby the proof.
\end{proof}

We are going to prove Theorem \ref{T:asczeromax}.

{\em Proof of Theorem \ref{T:asczeromax}}. An immediate consequence of Theorem \ref{T:bij} is an equidistribution of $(\asc,\zero,\max,\rmin)$ and $(\dist,\zero,\max,\rmin)$ on $\I_n$. That is equivalent to say that the quadruples $(\des,\lmax,\lmin,\rmax)$ and $(\ides,\lmax,\lmin,\rmax)$ have the same distribution on $\S_n$.

It remains to discuss the generating function. Let $G(t;u,q,w,z)$ be the generating function of inversion sequences counted by the length (variable $t$), the number of ascents ($\asc$, variable $u$), the number of zeros ($\zero$, variable $z$), the number of maximal entries ($\max$, variable $q$) and the last entry (variable $w$).
We set
\begin{align*}
G(t;q,u,w,z):=\sum_{n=1}^{\infty}t^n\sum_{s\in\I_n}u^{\asc(s)}q^{\max(s)}w^{\mathsf{last}(s)}z^{\zero(s)}.
\end{align*}
For any sequence $s\in\I_n$ with last entry $s_n=i$, we add a new element $j$ after the last entry $i$. We translate it into generating functions. if $i<j\le n$, then we substitute
\begin{align*}
w^i\rightarrow w^{i+1}+w^{i+2}+\cdots+qw^n=\frac{w(w^i-w^n)}{1-w}+(q-1)w^n
\end{align*}
in the generating function $utG(t;q,u,w,z)$. If $0\le j\le i$, then we substitute
\begin{align*}
w^i\rightarrow z+w+\cdots+w^{i}=\frac{1-w^{i+1}}{1-w}+z-1
\end{align*}
in the generating function $tG(t;q,u,w,z)$. Consequently we are led to
\begin{align*}
&\quad\frac{utw}{1-w}G(t;q,u,w,z)+ut\left(q-\frac{1}{1-w}\right)G(tw;q,u,1,z)\\
&+\left(z+\frac{w}{1-w}\right)tG(t;q,u,1,z)
-\frac{wt}{1-w}G(t;q,u,w,z)=G(t;q,u,w,z)-qtz,
\end{align*}
which can be simplified into
\begin{align}\label{E:ker1}
\nonumber\left(1-\frac{wt(u-1)}{1-w}\right)G(t;q,u,w,z)&=qtz+ut\left(q-\frac{1}{1-w}\right)G(tw;q,u,1,z)\\
&\quad+t\left(z+\frac{w}{1-w}\right)G(t;q,u,1,z).
\end{align}
We use kernel method to solve this equation. Choose
\begin{align*}
w=(1+ut-t)^{-1}
\end{align*}
so that the left hand side of (\ref{E:ker1}) becomes zero and the right hand side of (\ref{E:ker1}) becomes
\begin{align}\label{E:ker2}
G(t;q,u,1,z)=\frac{qtz(1-u)}{1+tz(u-1)}+\frac{u(1-(q-1)(u-1)t)}{1+zt(u-1)}G(tw;q,u,1,z).
\end{align}
Let $\sigma_m=t(1-mt+mut)^{-1}$. Then clearly $\sigma_0=t$ and $\sigma_1=tw$. Iterating (\ref{E:ker2}), we come to
\begin{align*}
\nonumber G(t;q,u,1,z)&=\sum_{m=0}^{n}\frac{qz(1-u)\sigma_m}{1-z(1-u)\sigma_m}
\prod_{i=0}^{m-1}
\left(\frac{u(1-(q-1)(u-1)\sigma_i)}{1-z(1-u)\sigma_i}\right)\\
&\quad+\prod_{j=0}^{n}\left(\frac{u(1-(q-1)(u-1)\sigma_j)}{1-z(1-u)\sigma_j}\right)G(\sigma_{n+1};q,u,1,z)
\end{align*}
Let $n\rightarrow \infty$, we conclude that
\begin{align*}
\nonumber G(t;q,u,1,z)&=\sum_{m=0}^{\infty}\frac{qz(1-u)\sigma_m u^m}{1-z(1-u)\sigma_m}\prod_{i=0}^{m-1}
\left(\frac{1-(q-1)(u-1)\sigma_i}{1-z(1-u)\sigma_i}\right),
\end{align*}
as a formal series in $u$ with rational coefficients in $t$. Replace $t$ by $t(1-u)^{-1}$, we obtain the generating function in Theorem \ref{T:asczeromax}. \qed

%Note that the above expression is only convergent as a series in $u$. In particular, if we set $u=1$, the result is zero because of the factor $(1-u)$. Equivalently, we can write

\section{Final remarks}\label{S:final}
We investigate three different decompositions of inversion sequences and derive bi-symmetric generating functions counted by the Euler-Stirling statistics. 
Most of the symmetric properties are discovered by combinatorial arguments. Two natural questions in this aspect are:
\begin{enumerate}
\item Can we prove $\CMcal{G}(t;x,u,v,q,1)=\CMcal{G}(t;x,u,q,v,1)=\CMcal{G}(t;u,x,q,v,1)$ by manipulations on the generating function (see the RHS of (\ref{E:thm1}))?
\item Recall that $\CMcal{F}(t;u,x,v,z)$ is defined in Remark \ref{Remark:2}, then it is suggested from numerical data by Maple that $\CMcal{F}(t;u,x,v,z)$ is bi-symmetric, i.e.,  $\CMcal{F}(t;u,x,v,z)=\CMcal{F}(t;x,u,z,v)$.
Can we prove/disprove it by the generating function formula or a bijection?
\item It would be interesting to find out a generating function with all five statistics included. Can one use the theory of $P$-partition to reach this goal?
\end{enumerate}

A conjecture on a quintuple equidistribution of Euler-Stirling statistics on permutations remains unsolved. Let us mention it below for interested readers.
\begin{conjecture}\cite{js}\label{OP:2}
	Let $\rep(s)=n-1-\dist(s)$ be the number of repeated entries of $s$.
	Then, there is a bijection $\Omega: \CMcal{I}_n\rightarrow \CMcal{I}_n$
	such that for all $s\in \CMcal{I}_n$, 
	\begin{equation*}
	(\asc,\rep,\zero,\max,\rmin)s=(\asc,\rep,\zero,\rmin,\max)\Omega(s).
	\end{equation*}
	Consequently for all $\pi\in \mathfrak{S}_n$,
	\begin{equation*}
	(\des,\iasc,\lmax,\lmin,\rmax)\pi=(\des,\iasc,\lmax,\rmax,\lmin)(\Theta^{-1}\circ \Omega\circ \Theta)(\pi),
	\end{equation*}
	where $\Theta:\mathfrak{S}_n\rightarrow \CMcal{I}_n$ is a bijection with the property (\ref{E:bv}).
\end{conjecture}

\section{Acknowledgement}
The author thanks Mireille Bousquet-M\'{e}lou, Zhicong Lin, Michael J. Schlosser and Huijie Shen for very helpful discussions. The author is very grateful to Michael J. Schlosser for providing an analytic proof of $\CMcal{G}(t;x,u,v,1,1)=\CMcal{G}(t;u,x,v,1,1)$ in Section \ref{S:4}.

%%%%%%%%%%%%%%%%%%%%%%%%%%%%%%%%%%%%%%%%%%%%%%%%%%%%%%%%%%%%%%%%%%%%%%%

%%%%%%%%%%%%%%%%%%%%%%%%%%%%%%%%%%%%%%%%%%%%%%%%%%%%%%%%%%%%%%%%%%%%%%%

\end{document}